\documentclass[a4paper]{amsart}
\usepackage[utf8]{inputenc}
\usepackage[T1]{fontenc}
\usepackage{amssymb,tikz-cd,diagbox,multirow}

\numberwithin{equation}{section}

\newtheorem{thm}{Theorem}[section]
\newtheorem{prop}[thm]{Proposition}
\newtheorem{lem}[thm]{Lemma}
\newtheorem{cor}[thm]{Corollary}

\theoremstyle{definition}
\newtheorem{defn}[thm]{Definition}

\theoremstyle{remark}
\newtheorem{rem}[thm]{Remark}

\newcommand{\Z}{\mathbb{Z}}

\newcommand{\R}{\mathbb{R}}
\newcommand{\aff}{\mathrm{aff}}

\newcommand{\St}{\mathrm{St}}
\newcommand{\coroot}[1]{#1^\vee}
\DeclareMathOperator{\Ker}{Ker}

\DeclareMathOperator{\Hom}{Hom}

\DeclareMathOperator{\Ext}{Ext}

\DeclareMathOperator{\Ind}{Ind}

\DeclareMathOperator{\Stab}{Stab}
\newcommand{\trivrep}{\boldsymbol{1}}

\newcommand{\Refs}[1]{\mathrm{Ref}(#1)}
\DeclareMathOperator{\val}{val}
\newcommand*{\opposite}[1]{#1^{\mathrm{op}}}
\DeclareMathOperator{\Coker}{Coker}

\title{Extension between simple modules of pro-$p$-Iwahori Hecke algebras}
\author{Noriyuki Abe}
\address{Department of Mathematics, Hokkaido University, Kita 10, Nishi 8, Kita-Ku, Sapporo, Hokkaido, 060-0810, Japan}
\email{abenori@math.sci.hokudai.ac.jp}

\begin{document}
\begin{abstract}
We calculate the extension groups between simple modules of pro-$p$-Iwahori Hecke algebras.
\end{abstract}
\maketitle

\section{Introduction}
Let $F$ be a non-archimedean local filed of residue characteristic $p$ and $G$ a connected reductive group over $F$.
Motivated by the modulo $p$ Langlands program, we study the modulo $p$ representation theory of $G$.
As in the classical (the representations over the field of complex numbers), Hecke algebras are useful tools for the study of modulo $p$ representations.
Especially, a pro-$p$-Iwahori Hecke algebra which is attached to a pro-$p$-Iwahori subgroup $I(1)$ has an important role in the study. (One reason is that any non-zero modulo $p$ representation has a non-zero $I(1)$-fixed vector.)
For example, this algebra is one of the most important tool for the proof of the classification theorem \cite{MR3600042}.

We focus on the representation theory of pro-$p$-Iwahori Hecke algebra.
Since the simple modules are classified~\cite{arXiv:1406.1003_accepted,MR3263136,Vigneras-prop-III}, we study its homological properties.
The aim of this paper is to calculate the extension between simple modules.
Note that such calculation was used to calculate the extension between irreducible modulo $p$ representations of $G$ when $G = \mathrm{GL}_2(\mathbb{Q}_p)$~\cite{MR2667891}.

We explain our result.
For each standard parabolic subgroup $P$, let $\mathcal{H}_P$ be the pro-$p$-Iwahori Hecke algebra of the Levi subgroup of $P$.
Then for a module $\sigma$ of $\mathcal{H}_P$, we can consider: the parabolic induction $I_P(\sigma)$ which is an $\mathcal{H}$-module, a certain parabolic subgroup $P(\sigma)$ containing $P$, a generalized Steinberg module $\St_Q^{P(\sigma)}(\sigma)$ where $Q$ is a parabolic subgroup between $P$ and $P(\sigma)$.
By \cite{arXiv:1406.1003_accepted}, each simple module is constructed by three steps: (1) starting with a supersingular module $\sigma$ of $\mathcal{H}_P$ where $P$ is a parabolic subgroup, (2) take a generalized Steinberg module $\St_Q^{P(\sigma)}(\sigma)$ (3) and take a parabolic induction $I_{P(\sigma)}(\St_Q^{P(\sigma)}(\sigma))$. (We do not explain the detail of notation here.)
Our calculation follows these steps.
Let $\pi_1 = I_{P(\sigma_1)}(\St_{Q_1}^{P(\sigma_1)}(\sigma_1))$ and $\pi_2 = I_{P(\sigma_2)}(\St_{Q_2}^{P(\sigma_2)}(\sigma_2))$ be two simple modules here $\sigma_1$ (resp. $\sigma_2$) is a simple supersingular module of $\mathcal{H}_{P_1}$ (resp.\ $\mathcal{H}_{P_2}$).

(1)
By considering the central characters, the extension $\Ext^i_{\mathcal{H}}(\pi_1,\pi_2)$ is zero if $P_1\ne P_2$ (Lemma~\ref{lem:vanishing of extension by central character}).
Hence we may assume $P_1 = P_2$.
Set $P = P_1$.

(2)
We prove
\[
\Ext^i_{\mathcal{H}}(I_{P(\sigma_1)}(\St_{Q_1}^{P(\sigma_1)}(\sigma_1)),I_{P(\sigma_2)}(\St_{Q_2}^{P(\sigma_2)}(\sigma_2)))
\simeq
\Ext^i_{\mathcal{H}_{P'}}(\St_{Q'_1}^{P'}(\sigma_1),\St_{Q'_2}^{P'}(\sigma_2))
\]
for some $Q'_1,Q'_2$ and $P'$ (Proposition~\ref{prop:removing parabolic induction}).
For the proof, we use the adjoint functors of parabolic induction and results in \cite{arXiv:1612.01312}.
Hence it is sufficient to calculate the extension groups between generalized Steinberg modules.

(3)
We prove
\[
\Ext^i_{\mathcal{H}}(\St_{Q_1}(\sigma_1),\St_{Q_2}(\sigma_2))
\simeq
\Ext^{i - r}_{\mathcal{H}}(e(\sigma_1),e(\sigma_2))
\]
for some (explicitly given) $r\in\Z_{\ge 0}$ or $0$ (Theorem~\ref{thm:extension between steinbergs}) using some involutions on $\mathcal{H}$ and results in \cite{arXiv:1704.00407}.
Here $e(\sigma)$ is the extension of $\sigma$ to $\mathcal{H}$ (Definition~\ref{defn:extension of module}).

(4)
We prove
\[
\Ext^i_{\mathcal{H}}(e(\sigma_1),e(\sigma_2))
\simeq
\Ext^i_{\mathcal{H}_P/I}(\sigma_1,\sigma_2)
\]
for some ideal $I\subset \mathcal{H}_P$ which acts on $\sigma_1$ and $\sigma_2$ by zero.
We use results of Ollivier-Schneider~\cite{MR3249689} for the proof.
The algebra $\mathcal{H}_P/I$ is not a pro-$p$-Iwahori Hecke algebra attached to a connected reductive group but a generic algebra in the sense of Vign\'eras~\cite[4.3]{MR3484112}.
Hence it is sufficient to calculate the extensions between supersingular simple modules of a generic algebra.

(5)
Now let $\mathcal{H}$ be a generic algebra and $\pi_1,\pi_2$ be simple supersinglar modules.
The algebra has the following decomposition: $\mathcal{H} = \mathcal{H}^\aff\otimes_{C[Z_\kappa]}C[\Omega(1)]$.
Here $\mathcal{H}^\aff\subset \mathcal{H}$ is an algebra called `the affine subalgebra', $\Omega(1)$ is a certain commutative group acting on $\mathcal{H}^\aff$, $Z_\kappa$ is a normal subgroup of $\Omega(1)$ and we have an embedding $C[Z_\kappa]\hookrightarrow \mathcal{H}^\aff$ which is compatible with the action of $\Omega(1)$ on $\mathcal{H}^\aff$.
Set $\Omega = \Omega(1)/Z_\kappa$.
By this decomposition and Hochschild-Serre type spectral sequence, we have an exact sequence
\[
0\to H^1(\Omega,\Hom_{\mathcal{H}^\aff}(\pi_1,\pi_2))
\to \Ext^1_{\mathcal{H}}(\pi_1,\pi_2)
\to \Ext^1_{\mathcal{H}^\aff}(\pi_1,\pi_2)^{\Omega}.
\]
We prove that the last map is surjective (Theorem~\ref{thm:surjectivity of exact sequence for supersingulars}).

Therefore it is sufficient to calculate two groups: $H^1(\Omega,\Hom_{\mathcal{H}^\aff}(\pi_1,\pi_2))$ and $\Ext^1_{\mathcal{H}^\aff}(\pi_1,\pi_2)^{\Omega}$.
By the classification result of supersingular simple modules~\cite{MR3263136,Vigneras-prop-III}, the restriction of $\pi_1,\pi_2$ to $\mathcal{H}^\aff$ are the direct sum of characters of $\mathcal{H}^\aff$.
Hence $\Hom_{\mathcal{H}^\aff}(\pi_1,\pi_2)$ is easily described and with this description we can calculate $H^1(\Omega,\Hom_{\mathcal{H}^\aff}(\pi_1,\pi_2))$ using well-known calculation of group cohomologies.
Note that $\Omega$ is commutative.
We also calculate $\Ext^1_{\mathcal{H}^\aff}(\Xi_1,\Xi_2)$ where $\Xi_1,\Xi_2$ are characters of $\mathcal{H}^\aff$ (Proposition~\ref{prop:extension between supersingulars, for H_aff}) following the method of Fayers~\cite{MR2134290}.
This is also calculated by  Nadimpalli~\cite{arXiv:1703.03110}.
Using this description, we can calculate $\Ext^1_{\mathcal{H}^\aff}(\pi_1,\pi_2)^{\Omega}$ and this finish the calculation of extensions between simple $\mathcal{H}$-modules.

\subsection*{Acknowledgment}
I thank Karol Koziol for explaining his calculation of some extensions using the resolution of Ollivier-Schneider.
This was helpful for the study.
Most of this work was done during my pleasant stay at Institut de mathématiques de Jussieu.
The work is supported by JSPS KAKENHI Grant Number 26707001.

\section{Preliminaries}\label{sec:Preliminaries}
\subsection{Pro-$p$-Iwahori Hecke algebra}\label{subsec:Prop-p-Iwahori Hecke algebra}
Let $\mathcal{H}$ be a pro-$p$-Iwahori Hecke algebra over a commutative ring $C$~\cite{MR3484112}.
We study modules over $\mathcal{H}$ in this paper.
\emph{In this paper, a module means a right module.}
The algebra $\mathcal{H}$ is defined with a combinatorial data $(W_\aff,S_\aff,\Omega,W,W(1),Z_\kappa)$ and a parameter $(q,c)$.

We recall the definitions.
The data satisfy the following.
\begin{itemize}
\item $(W_\aff,S_\aff)$ is a Coxeter system.
\item $\Omega$ acts on $(W_\aff,S_\aff)$.
\item $W = W_\aff\rtimes \Omega$.
\item $Z_\kappa$ is a finite commutative group.
\item The group $W(1)$ is an extension of $W$ by $Z_\kappa$, namely we have an exact sequence $1\to Z_\kappa\to W(1)\to W\to 1$.
\end{itemize}
The subgroup $Z_\kappa$ is normal in $W(1)$.
Hence the conjugate action of $w\in W(1)$ induces an automorphism of $Z_\kappa$, hence of the group ring $C[Z_\kappa]$.
We denote it by $c\mapsto w\cdot c$.

Let $\Refs{W_\aff}$ be the set of reflections in $W_\aff$ and $\Refs{W_\aff(1)}$ the inverse image of $\Refs{W_\aff}$ in $W(1)$.
The parameter $(q,c)$ is maps $q\colon S_\aff\to C$ and $c\colon \Refs{W_\aff(1)}\to C[Z_\kappa]$ with the following conditions. (Here the image of $s$ by $q$ (resp.\ $c$) is denoted by $q_s$ (resp.\ $c_s$).)
\begin{itemize}
\item For $w\in W$ and $s\in S_\aff$, if $wsw^{-1}\in S_\aff$ then $q_{wsw^{-1}} = q_s$.
\item For $w\in W(1)$ and $s\in \Refs{W_\aff(1)}$, $c_{wsw^{-1}} = w\cdot c_s$.
\item For $s\in \Refs{W_\aff(1)}$ and $t\in Z_\kappa$, we have $c_{ts} = tc_s$.
\end{itemize}
Let $S_\aff(1)$ be the inverse image of $S_\aff$ in $W(1)$.
For $s\in S_\aff(1)$, we write $q_s$ for $q_{\bar{s}}$ where $\bar{s}\in S_\aff$ is the image of $s$.
The length function on $W_\aff$ is denoted by $\ell$ and its inflation to $W$ and $W(1)$ is also denoted by $\ell$.

The $C$-algebra $\mathcal{H}$ is a free $C$-module and has a basis $\{T_w\}_{w\in W(1)}$.
The multiplication is given by
\begin{itemize}
\item (Quadratic relations) $T_{s}^2 = q_sT_{s^2} + c_sT_s$ for $s\in S_\aff(1)$.
\item (Braid relations) $T_{vw} = T_vT_w$ if $\ell(vw) = \ell(v) + \ell(w)$.
\end{itemize}
We extend $q\colon S_\aff\to C$ to $q\colon W\to C$ as follows.
For $w\in W$, take $s_1,\dots,s_l$ and $u\in\Omega$ such that $w = s_1\dotsm s_l u$ and $l = \ell(w)$.
Then put $q_w = q_{s_1}\dotsm q_{s_l}$.
From the definition, we have $q_{w^{-1}} = q_w$.
We also put $q_w = q_{\overline{w}}$ for $w\in W(1)$ with the image $\overline{w}$ in $W$.

\subsection{The data from a group}\label{subsec:The data from a group}
Let $F$ be a non-archimedean local field, $\kappa$ its residue field, $p$ its residue characteristic and $G$ a connected reductive group over $F$.
We can get the data in the previous subsection from $G$ as follows.
See \cite{MR3484112}, especially 3.9 and 4.2 for the details.

Fix a maximal split torus $S$ and denote the centralizer of $S$ in $G$ by $Z$.
Let $Z^0$ be the unique parahoric subgroup of $Z$ and $Z(1)$ its pro-$p$ radical.
Then the group $W(1)$ (resp.\ $W$) is defined by $W(1) = N_G(Z)/Z(1)$ (resp.\ $W = N_G(Z)/Z^0$) where $N_G(Z)$ is the normalizer of $Z$ in $G$.
We also let $Z_\kappa = Z^0/Z(1)$.
Let $G'$ be the group generated by the unipotent radical of parabolic subgroups \cite[II.1]{MR3600042} and $W_\aff$ the image of $G'\cap N_G(Z)$ in $W$.
Then this is a Coxeter group.
Fix a set of simple reflections $S_\aff$.
The group $W$ has the natural length function and let $\Omega$ be the set of length zero elements in $W$.
Then we get the data $(W_\aff,S_\aff,\Omega,W,W(1),Z_\kappa)$.

Consider the apartment attached to $S$ and an alcove surrounded by the hyperplanes fixed by $S_\aff$.
Let $I(1)$ be the pro-$p$-Iwahori subgroup attached to this alcove.
Then with $q_s = \#(I(1)\widetilde{s}I(1)/I(1))$ for $s\in S_\aff$ with a lift $\widetilde{s}\in N_G(Z)$ and suitable $c_s$, the algebra $\mathcal{H}$ is isomorphic to the Hecke algebra attached to $(G,I(1))$ \cite[Proposition~4.4]{MR3484112}.

When the data comes from the group $G$, let $W_\aff(1)$ be the image of $G'\cap N_G(Z)$ in $W(1)$ and put $\mathcal{H}_\aff = \bigoplus_{w\in W_\aff(1)}CT_w$.
This is a subalgebra of $\mathcal{H}$.

\emph{In this paper, except Section~\ref{sec:Ext1 between supersingulars}, we assume that the data comes from a connected reductive group.}
\subsection{The root system and the Weyl groups}
Let $W_0 = N_G(Z)/Z$ be the finite Weyl group.
Then this is a quotient of $W$.
Recall that we have the alcove defining $I(1)$.
Fix a special point $\boldsymbol{x}_0$ from the border of this alcove.
Then $W_0\simeq \Stab_W\boldsymbol{x}_0$ and the inclusion $\Stab_W\boldsymbol{x}_0\hookrightarrow W$ is a splitting of the canonical projection $W\to W_0$.
Throughout this paper, we fix this special point and regard $W_0$ as a subgroup of $W$.
Set $S_0 = S_\aff\cap W_0\subset W$.
This is a set of simple reflections in $W_0$.
For each $w\in W_0$, we fix a representative $n_w\in W(1)$ such that $n_{w_1w_2} = n_{w_1}n_{w_2}$ if $\ell(w_1w_2) = \ell(w_1) + \ell(w_2)$.

The group $W_0$ is the Weyl group of the root system $\Sigma$ attached to $(G,S)$.
Our fixed alcove and special point give a positive system of $\Sigma$, denoted by $\Sigma^+$.
The set of simple roots is denoted by $\Delta$.
As usual, for $\alpha\in\Delta$, let $s_\alpha\in S_0$ be a simple reflection for $\alpha$.

The kernel of $W(1)\to W_0$ (resp.\ $W\to W_0$) is denoted by $\Lambda(1)$ (resp.\ $\Lambda$).
Then $Z_\kappa\subset \Lambda(1)$ and we have $\Lambda = \Lambda(1)/Z_\kappa$.
The group $\Lambda$ (resp.\ $\Lambda(1)$) is isomorphic to $Z/Z^0$ (resp.\ $Z/Z(1)$).
Any element in $W(1)$ can be uniquely written as $n_w \lambda$ where $w\in W_0$ and $\lambda\in \Lambda(1)$.
We have $W = W_0\ltimes \Lambda$.

\subsection{The map $\nu$}
The group $W$ acts on the apartment attached to $S$ and the action of $\Lambda$ is by the translation.
Since the group of translations of the apartment is $X_*(S)\otimes_{\Z}\R$, we have a group homomorphism $\nu\colon \Lambda\to X_*(S)\otimes_{\Z}\R$.
The compositions $\Lambda(1)\to \Lambda\to X_*(S)\otimes_{\Z}\R$ and $Z\to \Lambda\to X_*(S)\otimes_{\Z}\R$ are also denoted by $\nu$.
The homomorphism $\nu\colon Z\to X_*(S)\otimes_{\Z}\R\simeq \Hom_\Z(X^*(S),\R)$ is characterized by the following: For $t\in S$ and $\chi\in X^*(S)$, we have $\nu(t)(\chi) = -\val(\chi(t))$ where $\val$ is the normalized valuation of $F$.

We call $\lambda\in \Lambda(1)$ dominant (resp.\ anti-dominant) if $\nu(\lambda)$ is dominant (resp.\ anti-dominant).

Since the group $W_\aff$ is a Coxeter system, it has the Bruhat order denoted by $\le$.
For $w_1,w_2\in W_\aff$, we write $w_1 < w_2$ if there exists $u\in \Omega$ such that $w_1u,w_2u\in W_\aff$ and $w_1u < w_2u$.
Moreover, for $w_1,w_2\in W(1)$, we write $w_1 < w_2$ if $w_1\in W_{\aff}(1)w_2$ and $\overline{w}_1 < \overline{w}_2$ where $\overline{w}_1,\overline{w}_2$ are the image of $w_1,w_2$ in $W$, respectively.
We write $w_1\le w_2$ if $w_1 < w_2$ or $w_1 = w_2$.

\subsection{Other basis}
For $w\in W(1)$, take $s_1,\dotsm,s_l\in S_\aff(1)$ and $u\in W(1)$ such that $l = \ell(w)$, $\ell(u) = 0$ and $w = s_1\dotsm s_l u$.
Set $T_w^* = (T_{s_1} - c_{s_1})\dotsm (T_{s_l} - c_{s_l})T_u$.
Then this does not depend on the choice and $\{T_w^*\}_{w\in W(1)}$ is a basis of $\mathcal{H}$.
In $\mathcal{H}[q_s^{\pm 1}]$, we have $T_w^* = q_wT_{w^{-1}}^{-1}$.

For simplicity, we always assume that our commutative ring $C$ contains a square root of $q_s$ which is denoted by $q_s^{1/2}$ for $s\in S_\aff$.
For $w = s_1\dotsm s_lu$ where $\ell(w) = l$ and $\ell(u) = 0$, $q_w^{1/2} = q_{s_1}^{1/2}\dotsm q_{s_l}^{1/2}$ is a square root of $q_w$.
For a spherical orientation $o$, there is a basis $\{E_o(w)\}_{w\in W(1)}$ of $\mathcal{H}$ introduced in \cite[5]{MR3484112}.
This satisfies the following product formula~\cite[Theorem~5.25]{MR3484112}.
\begin{equation}\label{eq:product formula}
E_o(w_1)E_{o\cdot w_1}(w_2) = q_{w_1w_2}^{-1/2}q_{w_1}^{1/2}q_{w_2}^{1/2}E_o(w_1w_2).
\end{equation}
\begin{rem}
Since we do not assume that $q_s$ is invertible in $C$, $q_{w_1w_2}^{-1/2}q_{w_1}^{1/2}q_{w_2}^{1/2}$ does not make sense in a usual way.
See \cite[Remark~2.2]{arXiv:1612.01312}.
\end{rem}

\subsection{Parabolic induction}
Since we have a positive system $\Sigma^+$, we have the minimal parabolic subgroup $B$ with a Levi part $Z$.
In this paper, \emph{parabolic subgroups are always standard, namely containing $B$}.
Note that such parabolic subgroups correspond to subsets of $\Delta$.

Let $P$ be a parabolic subgroup.
Attached to the Levi part of $P$ containing $Z$, we have the data $(W_{\aff,P},S_{\aff,P},\Omega_P,W_P,W_P(1),Z_\kappa)$ and the parameters $(q_P,c_P)$.
Hence we have the algebra $\mathcal{H}_P$.
The parameter $c_P$ is given by the restriction of $c$, hence we denote it just by $c$.
The parameter $q_P$ is defined as in \cite[4.1]{arXiv:1406.1003_accepted}.

For the objects attached to this data, we add the suffix $P$.
We have the set of simple roots $\Delta_P$, the root system $\Sigma_P$ and its positive system $\Sigma_P^+$, the finite Weyl group $W_{0,P}$, the set of simple reflections $S_{0,P}\subset W_{0,P}$, the length function $\ell_P$ and the base $\{T^P_w\}_{w\in W_P(1)}$, $\{T^{P*}_w\}_{w\in W_P(1)}$ and $\{E^P_o(w)\}_{w\in W_P(1)}$ of $\mathcal{H}_P$.
Note that we have no $\Lambda_P$, $\Lambda_P(1)$ and $Z_{\kappa,P}$ since they are equal to $\Lambda$, $\Lambda(1)$ and $Z_\kappa$.

An element $n_w\lambda\in W_P(1)$ where $w\in W_{P,0}$ and $\lambda\in \Lambda(1)$ is called $P$-positive (resp.\ $P$-negative) if for any $\alpha\in\Sigma^+\setminus\Sigma^+_P$ we have $\langle \alpha,\nu(\lambda)\rangle \le 0$ (resp.\ $\langle \alpha,\nu(\lambda)\rangle\ge 0$).
Set $\mathcal{H}_P^+ = \bigoplus_w CT^P_w$ where $w\in W_P(1)$ runs $P$-positive elements and define $\mathcal{H}_P^-$ by the similar way.
Then these are subalgebras of $\mathcal{H}_P$.
The linear maps $j^\pm_P\colon \mathcal{H}_P^{\pm}\to \mathcal{H}$ and $j^{\pm *}_P\colon \mathcal{H}_P^{\pm}\to \mathcal{H}$ defined by $j^\pm_P(T_w^P) = T_w$ and $j^{\pm *}_P(T_w^{P*}) = T_w^*$ are algebra homomorphisms.
\begin{prop}[{\cite[Theorem~1.4]{MR3437789}}]\label{prop:localization as Levi subalgebra}
Let $\lambda_P^+$ (resp.\ $\lambda_P^-$) be in the center of $W_P(1)$ such that $\langle \alpha,\nu(\lambda_P^+)\rangle < 0$ (resp.\ $\langle \alpha,\nu(\lambda_P^-)\rangle > 0$) for all $\alpha\in\Sigma^+\setminus\Sigma_P^+$.
Then $T^P_{\lambda_P^+} = T^{P*}_{\lambda_P^+} = E^P_{o_{-,P}}(\lambda_P^+)$ (resp.\ $T^P_{\lambda_P^-} = T^{P*}_{\lambda_P^-} = E^P_{o_{-,P}}(\lambda_P^-)$) is in the center of $\mathcal{H}_P$ and we have $\mathcal{H}_P = \mathcal{H}_P^+E^P_{o_{-,P}}(\lambda_P^+)^{-1}$ (resp.\ $\mathcal{H}_P = \mathcal{H}_P^-E^P_{o_{-,P}}(\lambda_P^-)^{-1}$).
\end{prop}

Now for an $\mathcal{H}_P$-module $\sigma$, we define the parabolically induced module $I_P(\sigma)$ by
\[
I_P(\sigma) = \Hom_{(\mathcal{H}_P^-,j_P^{-*})}(\mathcal{H},\sigma).
\]
This satisfies:
\begin{itemize}
\item $I_P$ is an exact functor.
\item $I_P$ has the left adjoint functor $L_P$. The functor $L_P$ is exact.
\item $I_P$ has the right adjoint functor $R_P$.
\end{itemize}
For parabolic subgroups $P\subset Q$, we also defines $\mathcal{H}_P^{Q\pm}\subset \mathcal{H}_P$ and $j_P^{Q\pm}\colon \mathcal{H}_P^{Q\pm}\to \mathcal{H}_Q$ and $j_P^{Q\pm *}\colon \mathcal{H}_P^{Q\pm}\to \mathcal{H}_Q$.
This defines the parabolic induction $I_P^Q$ from the category of $\mathcal{H}_P$-modules to the category of $\mathcal{H}_Q$-modules.

\subsection{Twist by $n_{w_Gw_P}$}
For a parabolic subgroup $P$, let $w_P$ be the longest element in $W_{0,P}$.
In particular, $w_G$ is the longest element in $W_0$.
Let $P'$ be a parabolic subgroup corresponding to $-w_G(\Delta_P)$, in other words, $P' = n_{w_Gw_P}\opposite{P}n_{w_Gw_P}^{-1}$ where $\opposite{P}$ is the opposite parabolic subgroup of $P$ with respect to the Levi part of $P$ containing $Z$.
Set $n = n_{w_Gw_P}$.
Then the map $\opposite{P}\to P'$ defined by $p\mapsto npn^{-1}$ is an isomorphism which preserves the data used to define the pro-$p$-Iwahori Hecke algebras.
Hence $T^P_w\mapsto T^{P'}_{nwn^{-1}}$ gives an isomorphism $\mathcal{H}_P\to \mathcal{H}_{P'}$.
This sends $T_w^{P*}$ to $T_{nwn^{-1}}^{P'*}$ and $E^P_{o_{+,P}\cdot v}(w)$ to $E^{P'}_{o_{+,P'}\cdot nvn^{-1}}(nwn^{-1})$ where $v\in W_{0,P}$.

Let $\sigma$ be an $\mathcal{H}_P$-module.
Then we define an $\mathcal{H}_{P'}$-module $n_{w_Gw_P}\sigma$ via the pull-back of the above isomorphism: $(n_{w_Gw_P}\sigma)(T^{P'}_w) = \sigma(T^P_{n_{w_Gw_P}^{-1}wn_{w_Gw_P}})$.
For an $\mathcal{H}_{P'}$-module $\sigma'$, we define $n_{w_Gw_P}^{-1}\sigma'$ by $(n_{w_Gw_P}^{-1}\sigma')(T_w^P) = \sigma'(T_{n_{w_Gw_P}wn_{w_Gw_P}^{-1}}^{P'})$.

\subsection{The extension and the generalized Steinberg modiles}
Let $P$ be the parabolic subgroup and $\sigma$ an $\mathcal{H}_P$-module.
For $\alpha\in\Delta$, let $P_\alpha$ be a parabolic subgroup corresponding to $\Delta_P\cup \{\alpha\}$.
Then we define $\Delta(\sigma)\subset\Delta$ by
\begin{align*}
&\Delta(\sigma)\\& = \{\alpha\in\Delta\mid \langle \Delta_P,\coroot{\alpha}\rangle = 0,\ \text{$\sigma(T^P_\lambda) = 1$ for any $\lambda\in W_{\aff,P_\alpha}(1)\cap \Lambda(1)$}\}\cup \Delta_P.
\end{align*}
Let $P(\sigma)$ be the parabolic subgroup corresponding to $\Delta(\sigma)$.
\begin{prop}[{\cite[Theorem~3.6]{arXiv:1703.10384}}]
Let $\sigma$ be an $\mathcal{H}_P$-module and $Q$ a parabolic subgroup between $P$ and $P(\sigma)$.
Denote the parabolic subgroup corresponding to $\Delta_Q\setminus\Delta_P$ by $P_2$.
Then there exist a unique $\mathcal{H}_Q$-module $e_Q(\sigma)$ acting on the same space as $\sigma$ such that
\begin{itemize}
\item $e_Q(\sigma)(T_w^{Q*}) = \sigma(T_w^{P*})$ for any $w\in W_P(1)$.
\item $e_Q(\sigma)(T_w^{Q*}) = 1$ for any $w\in W_{\aff,P_2}(1)$.
\end{itemize}
\end{prop}

\begin{defn}\label{defn:extension of module}
We call $e_Q(\sigma)$ the extension of $\sigma$ to $\mathcal{H}_Q$.
\end{defn}

A typical example of the extension is the trivial representation $\trivrep = \trivrep_G$.
This is a one-dimensional $\mathcal{H}$-module defined by $\trivrep(T_w) = q_w$, or equivalently $\trivrep(T_w^*) = 1$.
We have $\Delta(\trivrep_P) = \{\alpha\in\Delta\mid \langle \Delta_P,\coroot{\alpha}\rangle = 0\}\cup \Delta_P$ and, if $Q$ is a parabolic subgroup between $P$ and $P(\trivrep_P)$, we have $e_Q(\trivrep_P) = \trivrep_Q$

Let $P(\sigma)\supset P_0\supset Q_1\supset Q\supset P$.
Then as in \cite[4.5]{arXiv:1406.1003_accepted}, we have $I^{P_0}_{Q_1}(e_{Q_1}(\sigma))\subset I^{P_0}_Q(e_Q(\sigma))$.
Define
\[
\St_Q^{P_0}(\sigma) = \Coker\left(\bigoplus_{Q_1\supsetneq Q}I_{Q_1}^{P_0}(e_{Q_1}(\sigma))\to I^{P_0}_Q(e_Q(\sigma))\right).
\]
When $P_0 = G$, we write $\St_Q(\sigma)$ and call it generalized Steinberg modules.

\subsection{Supersingular modules}\label{subsec:supersingulars}
Let $\mathcal{O}$ be a conjugacy class in $W(1)$ which is contained in $\Lambda(1)$.
For a spherical orientation $o$, set $z_\mathcal{O} = \sum_{\lambda\in \mathcal{O}}E_o(\lambda)$.
Then this does not depend on $o$ and $z_\mathcal{O}\in \mathcal{Z}$ where $\mathcal{Z}$ is the center of $\mathcal{H}$ \cite[Theorem~5.1]{Vigneras-prop-III}.
The length of $\lambda\in \mathcal{O}$ does not depend on $\lambda$.
We denote it by $\ell(\mathcal{O})$.
For $\lambda\in \Lambda(1)$ and $w\in W(1)$, we put $w\cdot \lambda = w\lambda w^{-1}$.
\begin{defn}
Let $\pi$ be an $\mathcal{H}$-module.
We call $\pi$ supersingular if there exists $n\in \Z_{>0}$ such that $\pi z_\mathcal{O}^n = 0$ for any $\mathcal{O}$ such that $\ell(\mathcal{O}) > 0$.
\end{defn}
\begin{rem}\label{rem:a small reformulation of supersingularity}
Since $\pi z_\mathcal{O}\subset \pi$ is a submodule, if $\pi$ is simple then $\pi$ is supersingular if and only if $\pi z_\mathcal{O} = 0$ for any $\mathcal{O}$ such that $\ell(\mathcal{O}) > 0$.
Let $\lambda\in \Lambda(1)$.
Then $\ell(\lambda) \ne 0$ if and only if $\langle \alpha,\nu(\lambda)\rangle \ne 0$ for some $\alpha\in\Sigma$~\cite[Lemma~2.12]{arXiv:1612.01312}.
Hence a simple $\mathcal{H}$-module $\pi$ is supersingular if and only if $\pi(z_{W(1)\cdot \lambda}) = 0$ for any $\lambda$ such that $\langle \alpha,\nu(\lambda)\rangle\ne 0$ for some $\alpha\in\Sigma$.
\end{rem}

The simple supersingular $\mathcal{H}$-modules are classified in \cite{MR3263136,Vigneras-prop-III}.
We recall their results.
Let $W^\aff(1)$ be the inverse image of $W_\aff$ in $W(1)$.
\begin{rem}
When we do not assume that the data comes from a group, we have no $W_\aff(1)$ but we have $W^\aff(1)$.
Even though the data comes from a group, $W_\aff(1)$ is not equal to $W^\aff(1)$.
We have $Z_\kappa\subset W^\aff(1)$, however $Z_\kappa\not\subset W_\aff(1)$ in general.
Since we will not assume that the data comes from a group, we do not use $W_\aff(1)$ here.
\end{rem}
Put $\mathcal{H}^\aff = \bigoplus_{w\in W^\aff(1)}CT_w$.
Let $\chi$ be a character of $Z_\kappa$ and put $S_{\aff,\chi} = \{s\in S_\aff\mid \chi(c_{\widetilde{s}}) \ne 0\}$ where $\widetilde{s}\in W(1)$ is a lift of $s\in S_\aff$.
Note that if $\widetilde{s}'$ is another lift, then $\widetilde{s}' = t\widetilde{s}$ for some $t\in Z_\kappa$.
Hence $\chi(c_{\widetilde{s}'}) = \chi(t)\chi(c_{\widetilde{s}})$.
Therefore the condition does not depend on a choice of a lift.
Let $J\subset S_{\aff,\chi}$.
Then the character $\Xi = \Xi_{J,\chi}$ of $\mathcal{H}^\aff$ is defined by
\begin{align*}
\Xi_{J,\chi}(T_t) & = \chi(t) \quad (t\in Z_\kappa),\\
\Xi_{J,\chi}(T_{\widetilde{s}}) & = 
\begin{cases}
\chi(c_{\widetilde{s}}) & (s\in S_{\aff,\chi}\setminus J)\\
0 & (s\notin S_{\aff,\chi}\setminus J)
\end{cases}
=
\begin{cases}
\chi(c_{\widetilde{s}}) & (s\notin J),\\
0 & (s\in J).
\end{cases}
\end{align*}
where $\widetilde{s}\in W^\aff(1)$ is a lift of $s$ and the last equality easily follows from the definition of $S_{\aff,\chi}$.
Let $\Omega(1)_{\Xi}$ be the stabilizer of $\Xi$ and $V$ a simple $C[\Omega(1)_{\Xi}]$-module such that $V|_{Z_\kappa}$ is a direct sum of $\chi$.
Put $\mathcal{H}_\Xi = \mathcal{H}^\aff C[\Omega(1)_{\Xi}]$.
This is a subalgebra of $\mathcal{H}$.
For $X\in \mathcal{H}^\aff$ and $Y\in C[\Omega(1)_{\Xi}]$, we define the action of $XY$ on $\Xi\otimes V$ by $x\otimes y\mapsto xX\otimes yY$.
Then this defines a well-defined action of $\mathcal{H}_\Xi$ on $\Xi\otimes V$.
Set $\pi_{\chi,J,V} = (\Xi\otimes V)\otimes_{\mathcal{H}_\Xi}\mathcal{H}$.
\begin{prop}[{\cite[Theorem~1.6]{Vigneras-prop-III}}]
The module $\pi_{\chi,J,V}$ is simple and it is supersingular if and only if the groups generated by $J$ and generated by $S_{\aff,\chi}\setminus J$ are both finite.
If $C$ is an algebraically closed field, then any simple supersingular modules are given in this way.
\end{prop}

The construction of $\pi_{\chi,J,V}$ is still valid even if we do not assume that the data comes from a group.
In Section~\ref{sec:Ext1 between supersingulars}, we do not assume it and we calculate the extension between the modules constructed as above.

\subsection{Simple modules}
Assume that $C$ is an algebraically closed field of characteristic $p$.
We consider the following triple $(P,\sigma,Q)$.
\begin{itemize}
\item $P$ is a parabolic subgroup.
\item $\sigma$ is a simple supersingular $\mathcal{H}_P$-module.
\item $Q$ is a parabolic subgroup between $P$ and $P(\sigma)$.
\end{itemize}
Define
\[
I(P,\sigma,Q) = I_{P(\sigma)}(\St_Q^{P(\sigma)}(\sigma)).
\]
\begin{thm}[{\cite[Theorem~1.1]{arXiv:1406.1003_accepted}}]\label{thm:classification modulo supersingular}
The module $I(P,\sigma,Q)$ is simple and any simple module has this form.
Moreover, $(P,\sigma,Q)$ is unique up to isomorphism.
\end{thm}


\section{Reduction to supersingular representations}
Let $(P_1,\sigma_1,Q_1)$ and $(P_2,\sigma_2,Q_2)$ be triples as in Theorem~\ref{thm:classification modulo supersingular}.
\subsection{Central character}
We prove the following lemma.
\begin{lem}\label{lem:vanishing of extension by central character}
If $\Ext^i_{\mathcal{H}}(I(P_1,\sigma_1,Q_1),I(P_2,\sigma_2,Q_2))\ne 0$ for some $i\in\Z_{\ge 0}$, then $P_1 = P_2$.
\end{lem}
To prove this lemma, we calculate the action of the center $\mathcal{Z}$ on simple modules.
To do it, we need to calculate the action of $\mathcal{Z}$ on a parabolic induction.
\begin{lem}\label{lem:central character of parabolic induction}
Let $P$ be a parabolic subgroup, $\sigma$ a right $\mathcal{H}_{P}$-module.
For $W(1)$-orbit $\mathcal{O}$ in $\Lambda(1)$, set $\mathcal{O}_P = \{\lambda\in \mathcal{O}\mid \text{$\lambda$ is $P$-negative}\}$.
Then we have the following
\begin{enumerate}
\item The subset $\mathcal{O}_P\subset \Lambda(1)$ is $W_P(1)$-stable.
\item 
Let $\mathcal{O}_P = \mathcal{O}_1\cup \dots \cup \mathcal{O}_r$ be the decomposition into $W_P(1)$-orbits.
The action of $z_\mathcal{O}\in \mathcal{Z}$ on $I_P(\sigma)$ is induced by the action of $\sum_i z^P_{\mathcal{O}_i}$ on $\sigma$.
\end{enumerate}
\end{lem}
\begin{proof}
Since $\Sigma^+\setminus \Sigma_P^+$ is stable under the action of $W_{0,P}$, (1) follows from the definition of $P$-negative.

Let $\varphi\in I_P(\sigma) = \Hom_{(\mathcal{H}_P^-,j_P^{-*})}(\mathcal{H},\sigma)$.
Then for $X\in \mathcal{H}$, we have
\[
	(\varphi z_\mathcal{O})(X) = \varphi(z_\mathcal{O} X) = \varphi(Xz_\mathcal{O})
\]
since $z_\mathcal{O}$ is in the center of $\mathcal{H}$.
Hence by the definition of $z_\mathcal{O}$, we have
\[
	(\varphi z_\mathcal{O})(X)
	=
	\sum_{\lambda\in \mathcal{O}}\varphi(X E(\lambda))
	=
	\sum_i\sum_{\lambda\in \mathcal{O}_i} \varphi(X E(\lambda)) + \sum_{\lambda\in\mathcal{O},\ \text{not $P$-negative}}\varphi(XE(\lambda))
\]
We prove the vanishing of the second term.

Let $\lambda\in \mathcal{O}$ which is not $P$-negative.
Then there exists $\alpha\in\Sigma^+\setminus\Sigma_P^+$ such that $\langle \alpha,\nu(\lambda)\rangle < 0$.
Let $\lambda_P^-$ be as in Proposition~\ref{prop:localization as Levi subalgebra}.
Then $\langle \alpha,\nu(\lambda_P^-)\rangle > 0$.
Hence $\nu(\lambda)$ and $\nu(\lambda_P^-)$ does not belongs to the same closed Weyl chamber.
Therefore we have $E(\lambda)E(\lambda_P^-) = 0$ in $\mathcal{H}_C$ by \cite[(2.1), Lemma~2.11]{arXiv:1612.01312}.
Hence by \cite[Lemma~2.6]{arXiv:1612.01312},
\begin{align*}
\varphi(XE(\lambda)) & = \varphi(X E(\lambda))E^P(\lambda_P^-)E^P(\lambda_P^-)^{-1}\\
& = \varphi(X E(\lambda)j_P^{-*}(E^P(\lambda_P^-)))E^P(\lambda_P^-)^{-1}\\
& = \varphi(X E(\lambda)E(\lambda_P^-))E^P(\lambda_P^-)^{-1} = 0.
\end{align*}

If $\lambda\in \mathcal{O}_i$, then $E(\lambda)\in \mathcal{H}_P^-$.
Hence we have $E(\lambda) = j_P^{-*}(E^P(\lambda))$ by \cite[Lemma~2.6]{arXiv:1612.01312}.
Therefore
\[
\sum_{\lambda\in \mathcal{O}_i}\varphi(XE(\lambda)) = \varphi(X)\sum_{\lambda\in \mathcal{O}_i}\sigma(E^P(\lambda)) = \varphi(X)\sigma(z^P_{\mathcal{O}_i}).
\]
We get the lemma.
\end{proof}


\begin{lem}
Let $(P,\sigma,Q)$ be a triple as in Theorem~\ref{thm:classification modulo supersingular}.
Let $R$ be a parabolic subgroup and $\lambda = \lambda_R^-$ as in Proposition~\ref{prop:localization as Levi subalgebra}.
Then $z_{\mathcal{O}_\lambda} \ne 0$ on $I(P,\sigma,Q)$ if and only if $P\subset R$.
\end{lem}
\begin{proof}
Set $\mathcal{O} = \mathcal{O}_\lambda$.
Since $\Lambda(1)\subset W_R(1)$ and $\lambda$ is in the center of $W_R(1)$, $\lambda$ commutes with $\Lambda(1)$.
Hence $\mathcal{O} = \{n_w\cdot \lambda\mid w\in W_0\}$.

We prove that $W_P(1)$ acts transitively on $\mathcal{O}_P$.
Let $\mu\in \mathcal{O}_P$ and take $w\in W_0$ such that $\mu = n_w\cdot \lambda$.
Take $v\in W_{0,P}$ such that $v(\nu(\mu))$ is dominant with respect to $\Sigma^+_P$.
Since $v^{-1}(\Sigma^+\setminus\Sigma^+_P) = \Sigma^+\setminus\Sigma^+_P$ and $\mu$ is $P$-negative, we have $\langle v(\nu(\mu)),\alpha\rangle \ge 0$ for any $\alpha\in\Sigma^+\setminus\Sigma^+_P$.
Hence $v(\nu(\mu))$ is dominant.
Now $\nu(\lambda)$ and $v(\nu(\mu)) = vw(\nu(\lambda))$ is both dominant.
Hence $vw\in \Stab_{W_0}(\nu(\lambda)) = W_{0,R}$.
Since $\lambda$ is in the center of $W_R(1)$, we have $(n_vn_w)\cdot \lambda = \lambda$.
Hence $\mu = n_v^{-1}\cdot \lambda$.
Therefore $W_P(1)$ acts transitively on $\mathcal{O}_P$.

By the definition, $I(P,\sigma,Q)$ is a quotient of $I_{P(\sigma)}(I_Q^{P(\sigma)}(e_Q(\sigma))) = I_Q(e_Q(\sigma))$.
Moreover, by the definition of the extension, we have an embedding $e_Q(\sigma)\hookrightarrow I_P^Q(\sigma)$.
Hence we have $I_Q(e_Q(\sigma))\hookrightarrow I_Q(I_P^Q(\sigma)) = I_P(\sigma)$.
Let $\chi\colon \mathcal{Z}_P\to C$ be a central character of $\sigma$.
By the above lemma and the fact that $\mathcal{O}_P$ is a single $W_P(1)$-orbit, on $I_P(\sigma)$, $z_{\mathcal{O}_\lambda}$ acts by $\chi(z^P_{\mathcal{O}_P})$.
Since $\nu(\lambda)$ is dominant, $\lambda$ is $P$-negative.
Hence $\lambda\in \mathcal{O}_P$.
By the definition of supersingular representations with Remark~\ref{rem:a small reformulation of supersingularity}, $\chi(z_{\mathcal{O}_P}^P) = 0$ if and only if $\langle \alpha,\nu(\lambda)\rangle\ne 0$ for some $\alpha\in\Sigma_P^+$.
The condition on $\lambda = \lambda_R^-$ tells that $\langle \alpha,\nu(\lambda)\rangle\ne 0$ if and only if $\alpha\in\Sigma^+\setminus\Sigma^+_R$.
Therefore $\chi(z_{\mathcal{O}_P}^P)\ne 0$ if and only if $\Sigma^+_P\cap (\Sigma^+\setminus\Sigma_R^+) = \emptyset$ which is equivalent to $P\subset R$.
\end{proof}

\begin{proof}[Proof of Lemma~\ref{lem:vanishing of extension by central character}]
Assume that $P_1\ne P_2$.
Then we have $P_1\not\subset P_2$ or $P_1\not\supset P_2$.
Assume $P_1\not\subset P_2$ and take $\lambda = \lambda_{P_2}^-$ as in Proposition~\ref{prop:localization as Levi subalgebra}.
Put $\mathcal{O} = \{w\cdot \lambda\mid w\in W(1)\}$.
Then $z_{\mathcal{O}} = 0$ on $I(P_1,\sigma_1,Q_1)$ and $z_{\mathcal{O}} \ne 0$ on $I(P_2,\sigma_2,Q_2)$.
Hence the vanishing follows from a standard argument since $z_{\mathcal{O}}\in \mathcal{H}$ is in the center.
The case of $P_1\not\supset P_2$ is proved by the same way.
\end{proof}

\subsection{Reduction to generalized Steinberg modules}
By Lemma~\ref{lem:vanishing of extension by central character}, to calculate the extension between $I(P_1,\sigma_1,Q_1)$ and $I(P_2,\sigma_2,Q_2)$, we may assume $P_1 = P_2$.
We prove the following proposition.
\begin{prop}\label{prop:removing parabolic induction}
The extension group $\Ext^i_{\mathcal{H}}(I(P,\sigma_1,Q_1),I(P,\sigma,Q_2))$ is isomorphic to
\[
\Ext^i_{\mathcal{H}_{P(\sigma_1)\cap P(\sigma_2)}}(\St^{P(\sigma_1)\cap P(\sigma_2)}_{Q_1\cap P(\sigma_2)}(\sigma_1),\St^{P(\sigma_1)\cap P(\sigma_2)}_{Q_2}(\sigma_2)).
\]
if $Q_2\subset P(\sigma_1)$ and $\Delta(\sigma_1)\subset \Delta_{Q_1}\cup \Delta(\sigma_2)$.
Otherwise, the extension group is zero.
\end{prop}
Hence for the calculation of the extension, it is sufficient to calculate the extensions between generalized Steinberg modules.
For an $\mathcal{H}$-module $\pi$, set $\pi^* = \Hom_C(\pi,C)$.
The right $\mathcal{H}$-module structure on $\pi^*$ is given by $(fX)(v) = f(v\zeta(X))$ for $f\in \pi^*$, $v\in \pi$ and $X\in \mathcal{H}$.
Here the anti-involution $\zeta\colon \mathcal{H}\to \mathcal{H}$ is defined by $\zeta(T_w) = T_{w^{-1}}$.

\begin{lem}\label{lem:ext and dual}
We have $\Ext^i_{\mathcal{H}}(\pi_1,\pi_2^*)\simeq\Ext^i_{\mathcal{H}}(\pi_2,\pi_1^*)$.
In particular, if $\pi_1$ or $\pi_2$ is finite-dimensional, then $\Ext^i_{\mathcal{H}}(\pi_1,\pi_2)\simeq \Ext^i_{\mathcal{H}}(\pi_2^*,\pi_1^*)$.
\end{lem}
\begin{proof}
We have the isomorphism for $i = 0$ since both sides are equal to $\{f\colon \pi_1\times\pi_2\to C\mid f(x_1X,x_2) = f(x_1,\zeta(X)x_2)\ (x_1\in\pi_1,x_2\in\pi_2,X\in \mathcal{H})\}$.
Hence, in particular, if $\pi$ is projective, then $\pi^*$ is injective.
Let $\dots \to P_1\to P_0\to \pi_2\to 0$ be a projective resolution.
Then $\Ext^i_{\mathcal{H}}(\pi_2,\pi_1^*)$ is a $i$-th cohomology of the complex $\Hom(P_i,\pi_1^*)\simeq \Hom(\pi_1,P_i^*)$.
Since $0\to \pi_2^*\to P_0^*\to P_1^*\to \cdots$ is an injective resolution of $\pi_2^*$, this is $\Ext^i_{\mathcal{H}}(\pi_1,\pi_2^*)$.

If $\pi_2$ is finite-dimensional, then $\pi_2 \simeq (\pi_2^*)^*$.
Hence we have $\Ext^i_{\mathcal{H}}(\pi_1,\pi_2)\simeq\Ext^i_{\mathcal{H}}(\pi_1,(\pi_2^*)^*)\simeq\Ext^i_{\mathcal{H}}(\pi_2^*,\pi_1^*)$.
By the same argument, we have $\Ext^i_{\mathcal{H}}(\pi_1,\pi_2)\simeq \Ext^i_{\mathcal{H}}(\pi_2^*,\pi_1^*)$ if $\pi_1$ is finite-dimensional.
\end{proof}

\begin{prop}\label{prop:extension and right adjoint}
Let $P$ be a parabolic subgroup , $\pi$ an  $\mathcal{H}$-module and $\sigma$ an $\mathcal{H}_P$-module.
\begin{enumerate}
\item We have $\Ext^i_{\mathcal{H}}(\pi,I_P(\sigma))\simeq \Ext^i_{\mathcal{H}_P}(L_P(\pi),\sigma)$.
\item We have $\Ext^i_{\mathcal{H}}(I_P(\sigma),\pi^*)\simeq \Ext^i_{\mathcal{H}_P}(\sigma,R_P(\pi^*))$.
In particular, if $\pi$ is finite-dimensional, then $\Ext^i_{\mathcal{H}}(I_P(\sigma),\pi)\simeq \Ext^i_{\mathcal{H}_P}(\sigma,R_P(\pi))$.
\end{enumerate}
\end{prop}
\begin{proof}
The exactness of $I_P$ and $L_P$ implies (1).

Put $P' = n_{w_Gw_P}\opposite{P}n_{w_Gw_P}^{-1}$
Define the functor $I'_{P'}$ by
\[
I'_{P'}(\sigma') = \Hom_{(\mathcal{H}_{P'}^-,j_{P'}^-)}(\mathcal{H},\sigma')
\]
for an $\mathcal{H}_{P'}$-module $\sigma'$.
Then this has the left adjoint functor $L'_{P'}$ defined by $L'_{P'}(\pi) = \pi\otimes_{(\mathcal{H}_{P'}^-,j_{P'}^-)}\mathcal{H}_{P'}$.
This is exact since $\mathcal{H}_{P'}$ is a localization of $\mathcal{H}_{P'}^-$ by Proposition~\ref{prop:localization as Levi subalgebra}.
Set $\sigma_{\ell - \ell_P}(T^P_w) = (-1)^{\ell(w) - \ell_P(w)}\sigma(T_w)$~\cite[4.1]{arXiv:1612.01312}.
Using \cite[Proposition~4.2]{arXiv:1704.00407}, for an $\mathcal{H}_P$-module $\sigma$, we have
\begin{align*}
\Ext^i_{\mathcal{H}}(I_P(\sigma),\pi^*)
& \simeq\Ext^i_{\mathcal{H}}(\pi,I_P(\sigma)^*)\\
& \simeq\Ext^i_{\mathcal{H}}(\pi,I'_{P'}(n_{w_Gw_P}\sigma^*_{\ell - \ell_P}))\\
& \simeq \Ext^i_{\mathcal{H}_P}(n_{w_Gw_P}^{-1}L'_{P'}(\pi),\sigma^*_{\ell - \ell_P})\\
& \simeq\Ext^i_{\mathcal{H}_P}(\sigma_{\ell - \ell_P},n_{w_Gw_P}^{-1}L'_{P'}(\pi)^*)\\
& \simeq \Ext^i_{\mathcal{H}_P}(\sigma,(n_{w_Gw_P}^{-1}L'_{P'}(\pi))^*_{\ell - \ell_P}).
\end{align*}
Put $i = 0$.
Then we get $(n_{w_Gw_P}^{-1}L'_{P'}(\pi))^*_{\ell - \ell_P} \simeq R_P(\pi^*)$ by $\Hom_{\mathcal{H}}(I_P(\sigma),\pi^*)\simeq \Hom_{\mathcal{H}_P}(\sigma,R_P(\pi^*))$.
Hence we get (2).

If $\pi$ is finite-dimensional, then $\pi = (\pi^*)^*$.
Hence we get $\Ext^i_{\mathcal{H}}(I_P(\sigma),\pi)\simeq \Ext^i_{\mathcal{H}_P}(\sigma,R_P(\pi))$ applying (3) to $\pi^*$.
\end{proof}

\begin{proof}[Proof of Proposition~\ref{prop:removing parabolic induction}]
Since $I(P,\sigma_2,Q_2)$ is finite-dimensional, we have
\[
\Ext^i_{\mathcal{H}}(I(P,\sigma_1,Q_1),I(P,\sigma_2,Q_2)) 
= \Ext^i_{\mathcal{H}_{P(\sigma_1)}}(\St^{P(\sigma_1)}_{Q_1}(\sigma_1),R_{P(\sigma_1)}(I(P,\sigma_2,Q_2))).
\]
We have $R_{P(\sigma_1)}(I(P,\sigma_2,Q_2)) = 0$ if $Q_2\not\subset P(\sigma_1)$ by \cite[Theorem~5.20]{arXiv:1612.01312}.
If $Q_2\subset P(\sigma_1)$, then $R_{P(\sigma_1)}(I(P,\sigma_2,Q_2)) = I_{P(\sigma)}(P,\sigma,Q_2)$.
Hence the extension group is isomorphic to
\begin{align*}
& \Ext^i_{\mathcal{H}_{P(\sigma_1)}}(\St^{P(\sigma_1)}_{Q_1}(\sigma_1),I_{P(\sigma_1)}(P,\sigma_2,Q_2))\\
& = \Ext^i_{\mathcal{H}_{P(\sigma_1)}}(\St^{P(\sigma_1)}_{Q_1}(\sigma_1),I_{P(\sigma_2)\cap P(\sigma_1)}^{P(\sigma_1)}(\St_{Q_2}(\sigma_2)))\\
& = \Ext^i_{\mathcal{H}_{P(\sigma_1)\cap P(\sigma_2)}}(L_{P(\sigma_2)\cap P(\sigma_1)}^{P(\sigma_1)}(\St^{P(\sigma_1)}_{Q_1}(\sigma_1)),\St^{P(\sigma_1)\cap P(\sigma_2)}_{Q_2}(\sigma_2))
\end{align*}
We have $L_{P(\sigma_2)\cap P(\sigma_1)}^{P(\sigma_1)}(\St^{P(\sigma_1)}_{Q_1}(\sigma_1)) = 0$ if $\Delta(\sigma_1) \ne \Delta(Q_1)\cup \Delta_{P(\sigma_1)\cap P(\sigma_2)}$ or $P\not\subset P(\sigma_1)\cap P(\sigma_2)$ by \cite[Proposition~5.10, Proposition~5.18]{arXiv:1612.01312}.
If it is not zero, then the extension group is isomorphic to 
\[
\Ext^i_{\mathcal{H}_{P(\sigma_1)\cap P(\sigma_2)}}(\St^{P(\sigma_1)\cap P(\sigma_2)}_{Q_1\cap P(\sigma_2)}(\sigma_1),\St^{P(\sigma_1)\cap P(\sigma_2)}_{Q_2}(\sigma_2)).
\]
This holds if $Q_2\subset P(\sigma_1)$, $\Delta(\sigma_1) = \Delta(Q_1)\cup \Delta_{P(\sigma_1)\cap P(\sigma_2)}$ and $P\subset P(\sigma_1)\cap P(\sigma_2)$ and, otherwise the extension group is zero.
Note that we always have $P\subset P(\sigma_1)\cap P(\sigma_2)$ since both $P(\sigma_1)$ and $P(\sigma_2)$ contain $P$.
Since $Q_1\subset P(\sigma_1)$, $\Delta(Q_1)\cup \Delta_{P(\sigma_1)\cap P(\sigma_2)} = (\Delta(\sigma_1)\cap \Delta(Q_1))\cup (\Delta(\sigma_1)\cap \Delta(\sigma_2)) = \Delta(\sigma_1)\cap (\Delta_{Q_1}\cup \Delta(\sigma_2))$. (Recall that $P(\sigma_1)$ is the parabolic subgroup corresponding to $\Delta(\sigma_1)$.)
Hence we have $\Delta(\sigma_1) = \Delta(Q_1)\cup \Delta(P(\sigma_1)\cap P(\sigma_2))$ if and only if $\Delta(\sigma_1)\subset \Delta_{Q_1}\cup \Delta(\sigma_2)$.
We get the proposition.
\end{proof}
Therefore, to calculate the extension groups, we may assume $P(\sigma_1) = P(\sigma_2) = G$.

\subsection{Extensions between generalized Steinberg modules}
We assume that $P(\sigma_1) = P(\sigma_2) = G$ and we continue the calculation of the extension groups.
\begin{lem}\label{lem:Ext of Steinberg, first step}
Let $Q_{11},Q_{12},Q_2$ be parabolic subgroups and $\alpha\in\Delta_{Q_{12}}$ such that $\Delta_{Q_{11}} = \Delta_{Q_{12}}\setminus\{\alpha\}$.
Then we have
\[
\Ext^i_{\mathcal{H}}(\St_{Q_{11}}(\sigma_1),\St_{Q_2}(\sigma_2))
\simeq
\begin{cases}
\Ext^{i - 1}_{\mathcal{H}}(\St_{Q_{12}}(\sigma_1),\St_{Q_2}(\sigma_2)) & (\alpha\in\Delta_{Q_2}),\\
\Ext^{i + 1}_{\mathcal{H}}(\St_{Q_{12}}(\sigma_1),\St_{Q_2}(\sigma_2)) & (\alpha\notin\Delta_{Q_2}).
\end{cases}
\]
\end{lem}
\begin{proof}
Let $P_1$ be a parabolic subgroup corresponding to $\Delta\setminus\{\alpha\}$.
First we prove that there exists an exact sequence
\begin{equation}\label{eq:rank 1 induction}
0\to \St_{Q_{12}}(\sigma_1)\to I_{P_1}(\St_{Q_{11}}^{P_1}(\sigma_1))\to \St_{Q_{11}}(\sigma_1)\to 0.
\end{equation}

We start with the following exact sequence.
\[
0\to \sum_{P_1\supset Q\supsetneq Q_{11}}I^{P_1}_{Q}(e_{Q}(\sigma_1))\to I_{Q_{11}}^{P_1}(e_{Q_{11}}(\sigma_1))\to \St_{Q_{11}}^{P_1}(\sigma_1)\to 0,
\]
Apply $I_{P_1}$ to this exact sequence.
Then we have
\[
0\to \sum_{P_1\supset Q\supsetneq Q_{11}}I_{Q}(e_{Q}(\sigma_1))\to I_{Q_{11}}(e_{Q_{11}}(\sigma_1))\to I_{P_1}(\St_{Q_{11}}^{P_1}(\sigma_1))\to 0.
\]
Hence we get the following commutative diagram with exact columns:
\[
\begin{tikzcd}
0\arrow{d} & 0\arrow{d}\\
\sum_{P_1\supset Q\supsetneq Q_{11}}I_{Q}(e_{Q}(\sigma_1))\arrow{d}\arrow[hookrightarrow]{r} & \sum_{Q\supsetneq Q_{11}}I_{Q}(e_Q(\sigma))\arrow{d}\\
I_{Q_{11}}(e_{Q_{11}}(\sigma_1)) \arrow[-,double line]{r}\arrow{d}& I_{Q_{11}}(e_{Q_{11}}(\sigma_1))\arrow{d}\\
I_{P_1}(\St_{Q_{11}}^{P_1}(\sigma_1))\arrow{r}\arrow{d} & \St_{Q_{11}}(\sigma_1)\arrow{d}\\
0 & 0
\end{tikzcd}
\]
Hence $I_{P_1}(\St_{Q_{11}}^{P_1}(\sigma_1))\to \St_{Q_{11}}(\sigma_1)$ is surjective and the kernel is isomorphic to
\[
\left.\sum_{Q\supsetneq Q_{11}}I_{Q}(e_Q(\sigma))\middle/\sum_{P_1\supset Q\supsetneq Q_{11}}I_{Q}(e_{Q}(\sigma_1))\right.
\]
by the snake lemma.

We prove:
\begin{enumerate}
\item $I_{Q_{12}}(e_{Q_{12}}(\sigma_1)) + \sum_{P_1\supset Q\supsetneq Q_{11}}I_{Q}(e_{Q}(\sigma_1)) = \sum_{Q\supsetneq Q_{11}}I_{Q}(e_Q(\sigma))$.
\item $I_{Q_{12}}(e_{Q_{12}}(\sigma_1)) \cap \sum_{P_1\supset Q\supsetneq Q_{11}}I_{Q}(e_{Q}(\sigma_1)) = \sum_{Q\supsetneq Q_{12}}I_{Q}(e_Q(\sigma))$.
\end{enumerate}
We prove (1).
Since $Q_{12}\supsetneq Q_{11}$, $I_{Q_{12}}(e_{Q_{12}}(\sigma_1))$ is contained in the right hand side.
Obviously $\sum_{P_1\supset Q\supsetneq Q_{11}}I_{Q}(e_{Q}(\sigma_1))$ is also contained in the right hand side.
Hence $I_{Q_{12}}(e_{Q_{12}}(\sigma_1)) + \sum_{P_1\supset Q\supsetneq Q_{11}}I_{Q}(e_{Q}(\sigma_1)) \subset \sum_{Q\supsetneq Q_{11}}I_{Q}(e_Q(\sigma))$.
Take $Q_\supsetneq Q_{11}$ and we prove that $I_Q(e_Q(\sigma_1))\subset I_{Q_{12}}(e_{Q_{12}}(\sigma_1)) + \sum_{P_1\supset Q\supsetneq Q_{11}}I_{Q}(e_{Q}(\sigma_1))$.
If $P_1\supset Q$, then it is obvious.
We assume that $P_1\not\supset Q$.
Since $\Delta_{P_1} = \Delta\setminus\{\alpha\}$, this is equivalent to $\alpha\in\Delta_Q$.
Hence $\Delta_Q\supset \Delta_{Q_{11}}\cup \{\alpha\} = \Delta_{Q_{12}}$.
Therefore we have $Q\supset Q_{12}$.
Hence $I_{Q}(e_{Q}(\sigma_1))\subset I_{Q_{12}}(e_{Q_{12}}(\sigma_1))$.

We prove (2).
By \cite[Lemma~3.10]{arXiv:1704.00407}, the left hand side is 
\[
\sum_{P_1\supset Q\supsetneq Q_{11}}I_{\langle Q,Q_{12}\rangle}(e_{\langle Q,Q_{12}\rangle}(\sigma_1))
\]
where $\langle Q,Q_{12}\rangle$ is the subgroup generated by $Q$ and $Q_{12}$.
We prove
\[
\{\langle Q,Q_{12}\rangle\mid P_1\supset Q\supsetneq Q_{11}\} = \{Q\mid Q\supsetneq Q_{12}\}.
\]
If $Q$ satisfies $P_1\supset Q\supsetneq Q_{11}$, then there exists $\beta\in \Delta_Q\setminus \Delta_{Q_{11}}$.
We have $\beta\in \Delta_Q\subset \Delta_{P_1} = \Delta\setminus\{\alpha\}$.
Therefore we have $\beta\ne\alpha$.
Hence $\beta\notin\Delta_{Q_{11}}\cup \{\alpha\} = \Delta_{Q_{12}}$.
On the other hand, $\beta\in\Delta_Q\subset \Delta_{\langle Q,Q_{12}\rangle}$.
Namely we have $\beta\in\Delta_{\langle Q,Q_{12}\rangle}\setminus\Delta_{Q_{12}}$.
Obviously $\langle Q,Q_{12}\rangle\supset Q_{12}$.
Therefore we get $\langle Q,Q_{12}\rangle\supsetneq Q_{12}$.

On the other hand, assume that $Q\supsetneq Q_{12}$.
Then $\alpha\in \Delta_Q$ since $\alpha\in\Delta_{Q_{12}}$.
Let $Q'$ be the parabolic subgroup corresponding to $\Delta_Q\setminus\{\alpha\}$.
Then we have $\Delta_{Q'}\subset \Delta\setminus\{\alpha\} = \Delta_{P_1}$ and $\Delta_{Q'} = \Delta_Q\setminus\{\alpha\} \supsetneq \Delta_{Q_{12}}\setminus\{\alpha\} = \Delta_{Q_{11}}$.
Hence $P_1\supset Q'\supsetneq Q_{11}$.
We have $\Delta_{\langle Q',Q_{12}\rangle} = \Delta_{Q'}\cup \Delta_{Q_{12}} = (\Delta_Q\setminus\{\alpha\})\cup\Delta_{Q_{11}}\cup\{\alpha\} = \Delta_Q\cup\Delta_{Q_{11}}\cup\{\alpha\}$.
This is $\Delta_Q$ since $\Delta_Q\supset \Delta_{Q_{12}} = \Delta_{Q_{11}}\cup\{\alpha\}$.
Hence $Q = \langle Q',Q_{12}\rangle$.
We get the existence of the exact sequence~\eqref{eq:rank 1 induction}.

Assume that $\alpha\in\Delta_{Q_2}$.
Then $\alpha\in \Delta_{Q_2}$ and $\alpha\notin\Delta_{Q_2\cap P_1}$.
Hence $\Delta_{Q_2}\ne\Delta_{Q_2\cap P_1}\cup \Delta_P$.
Therefore $R_{P_1}(\St_{Q_2}(\sigma_2)) = 0$ by \cite[Proposition~5.11]{arXiv:1612.01312}.
We have an exact sequence
\begin{align*}
& \Ext^i_{\mathcal{H}}(I_{P_1}(\St_{Q_{11}}^{P_1}(\sigma_1)),\St_{Q_2}(\sigma_2))\to
\Ext^i_{\mathcal{H}}(\St_{Q_{12}}(\sigma_1),\St_{Q_2}(\sigma_2))\\
& \to \Ext^{i + 1}_{\mathcal{H}}(\St_{Q_{11}}(\sigma_1),\St_{Q_2}(\sigma_2))\to
\Ext^{i + 1}_{\mathcal{H}}(I_{P_1}(\St_{Q_{11}}^{P_1}(\sigma_1)),\St_{Q_2}(\sigma_2)).
\end{align*}
Since $R_{P_1}(\St_{Q_2}(\sigma_2)) = 0$, for any $j$ we have
\[
\Ext^j_{\mathcal{H}}(I_{P_1}(\St_{Q_{12}}(\sigma_1)),\St_{Q_2}(\sigma_2)) = \Ext^j_{\mathcal{H}}(\St_{Q_{12}}(\sigma_1),R_{P_1}(\St_{Q_2}(\sigma_2))) = 0.
\]
Therefore we get 
\[
\Ext^i_{\mathcal{H}}(\St_{Q_{12}}(\sigma_1),\St_{Q_2}(\sigma_2))
\simeq \Ext^{i + 1}_{\mathcal{H}}(\St_{Q_{11}}(\sigma_1),\St_{Q_2}(\sigma_2)).
\]

Next assume that $\alpha\notin\Delta_{Q_2}$.
Let $Q_{11}^c$ (resp,\ $Q_{12}^c,Q_2^c$) be the parabolic subgroup corresponding to $(\Delta\setminus\Delta_{Q_{11}})\cup\Delta_P$ (resp.\ $(\Delta\setminus\Delta_{Q_{12}})\cup\Delta_P,(\Delta\setminus\Delta_{Q_2})\cup\Delta_P$).
Let $\iota = \iota_G\colon \mathcal{H}\to \mathcal{H}$ be the involution defined by $\iota(T_w) = (-1)^{\ell(w)}T_w^*$ and set $\pi^\iota = \pi\circ \iota$ for an $\mathcal{H}$-module $\pi$.
Then we have
\begin{align*}
\Ext^i_{\mathcal{H}}(\St_{Q_{11}}(\sigma_1),\St_{Q_2}(\sigma_2))& \simeq
\Ext^i_{\mathcal{H}}((\St_{Q_{11}}(\sigma_1))^{\iota},(\St_{Q_2}(\sigma_2))^{\iota})\\
&\simeq  \Ext^i_{\mathcal{H}}(\St_{Q_{11}^c}(\sigma_{1,\ell - \ell_P}^{\iota_P}),\St_{Q_2^c}(\sigma_{2,\ell - \ell_P}^{\iota_P}))
\end{align*}
by \cite[Theorem~3.6]{arXiv:1704.00407}.
Now we have $\alpha\in \Delta_{Q_2^c}$.
Applying the lemma (where $Q_{11} = Q_{12}^c$ and $Q_{12} = Q_{11}^c$), we have
\begin{align*}
\Ext^i_{\mathcal{H}}(\St_{Q_{11}^c}(\sigma_{1,\ell - \ell_P}^{\iota_P}),\St_{Q_2^c}(\sigma_{2,\ell - \ell_P}^{\iota_P}))
& \simeq \Ext^{i - 1}_{\mathcal{H}}(\St_{Q_{12}^c}(\sigma_{1,\ell - \ell_P}^{\iota_P}),\St_{Q_2^c}(\sigma_{2,\ell - \ell_P}^{\iota_P}))\\
& \simeq \Ext^{i - 1}_{\mathcal{H}}((\St_{Q_{12}}(\sigma_1))^\iota,(\St_{Q_2}(\sigma_2))^\iota)\\
& \simeq \Ext^{i - 1}_{\mathcal{H}}(\St_{Q_{12}}(\sigma_1),\St_{Q_2}(\sigma_2)).
\end{align*}
We get the lemma.
\end{proof}

For sets $X,Y$, let $X\bigtriangleup Y = (X\setminus Y)\cup (Y\setminus X)$ be the symmetric difference.
\begin{thm}\label{thm:extension between steinbergs}
We have
\[
\Ext^i_{\mathcal{H}}(\St_{Q_1}(\sigma_1),\St_{Q_2}(\sigma_2)) \simeq \Ext^{i - \#(\Delta_{Q_1}\bigtriangleup\Delta_{Q_2})}_{\mathcal{H}}(e_G(\sigma_1),e_G(\sigma_2)).
\]
\end{thm}
\begin{proof}
By applying Lemm~\ref{lem:Ext of Steinberg, first step} several times, we have
\[
\Ext^i_{\mathcal{H}}(\St_{Q_1}(\sigma_1),\St_{Q_2}(\sigma_2))
\simeq \Ext^{i - r_1}_{\mathcal{H}}(e_G(\sigma_1),\St_{Q_2}(\sigma_2))
\]
where $r_1 = \#\{\alpha\in\Delta\setminus \Delta_{Q_1}\mid \alpha\in\Delta_{Q_2}\} - \#\{\alpha\in\Delta\setminus\Delta_{Q_1}\mid \alpha\notin\Delta_{Q_2}\}$.
Set $Q'_2 = n_{w_Gw_{Q_2}}\opposite{Q_2}n_{w_Gw_{Q_2}}^{-1}$.
Then by Lemma~\ref{lem:ext and dual}, we get
\begin{align*}
\Ext^{i - r_1}_{\mathcal{H}}(e_G(\sigma_1),\St_{Q_2}(\sigma_2))
& \simeq \Ext^{i - r_1}_{\mathcal{H}}((\St_{Q_2}(\sigma_2))^*,e_G(\sigma_1)^*)\\
& \simeq \Ext^{i - r_1}_{\mathcal{H}}(\St_{Q'_2}(\sigma_2^*),e_G(\sigma_1^*)).
\end{align*}
Again using Lemma~\ref{lem:Ext of Steinberg, first step}, we have
\[
\Ext^{i - r_1}_{\mathcal{H}}(\St_{Q'_2}(\sigma_2^*),e_G(\sigma_1^*))
\simeq
\Ext^{i - r_1 - r_2}_{\mathcal{H}}(e_G(\sigma_2^*),e_G(\sigma_1^*)).
\]
where $r_2 = \#(\Delta\setminus\Delta_{Q'_2}) = \#(\Delta\setminus\Delta_{Q_2})$.
Applying Lemma~\ref{lem:ext and dual} again, we get
\[
\Ext^{i - r_1 - r_2}_{\mathcal{H}}(e_G(\sigma_2^*),e_G(\sigma_1^*))
\simeq
\Ext^{i - r_1 - r_2}_{\mathcal{H}}(e_G(\sigma_1),e_G(\sigma_2)).
\]
Since $r_1 + r_2 = \#(\Delta_{Q_1}\bigtriangleup\Delta_{Q_2})$, we get the lemma.
\end{proof}

Recall that the trivial module $\trivrep$ is defined by $\trivrep(T_w) = q_w$.
We denote the restriction of $\trivrep$ to $\mathcal{H}_\aff$ by $\trivrep_{\mathcal{H}_\aff}$.
\begin{cor}\label{cor:vanishing of ext between trivials}
We have $\Ext^i_{\mathcal{H}_\aff}(\trivrep_{\mathcal{H}_\aff},\trivrep_{\mathcal{H}_\aff}) = 0$ for $i > 0$.
\end{cor}
\begin{proof}

Let $\overline{\mathcal{H}}_\aff$ be the quotient of $\mathcal{H}_\aff$ by the ideal generated by $\{T_t - 1\mid t\in Z_\kappa\cap W_\aff(1)\}$.
Then this is the Hecke algebra attached to the Coxeter system $(W_\aff,S_\aff)$.
Let $\pi_2$ (resp.\ $\overline{\pi}_1$) be an an $\mathcal{H}_\aff$-module (resp.\ $\overline{\mathcal{H}}_\aff$-module).
Then we have $\Hom_{\mathcal{H}_\aff}(\overline{\pi}_1,\pi_2) = \Hom_{\overline{\mathcal{H}}_\aff}(\overline{\pi}_1,\pi_2^{Z_\kappa})$.
In particular $\pi_2\mapsto \pi_2^{Z_\kappa}$ sends injective $\mathcal{H}$-modules to injective $\overline{\mathcal{H}}_\aff$-modules.
Since the functor $\pi_2\mapsto \pi_2^{Z_\kappa}$ is exact, we have $\Ext^i_{\mathcal{H}_\aff}(\overline{\pi}_1,\pi_2) \simeq \Ext^i_{\overline{\mathcal{H}}_\aff}(\overline{\pi}_1,\pi_2^{Z_\kappa})$.
Therefore we have $\Ext^i_{\mathcal{H}_\aff}(\trivrep_{\mathcal{H}_\aff},\trivrep_{\mathcal{H}_\aff})\simeq \Ext^i_{\overline{\mathcal{H}}_\aff}(\trivrep_{\overline{\mathcal{H}}_\aff},\trivrep_{\overline{\mathcal{H}}_\aff})$.
Consider the root system which defines $(W_\aff,S_\aff)$ and let $H$ be the split simply-connected semisimple group with this root system.
Then the affine Hecke algebra attached to $H$ is $\overline{\mathcal{H}}_\aff$.
Let $\mathcal{H}'$ be the pro-$p$-Iwahori Hecke algebra for $H$.
Then $H' = H$ by \cite[II.3.Proposition]{MR3600042}, hence $\mathcal{H}'_\aff = \mathcal{H}'$.
Therefore the above argument implies that $\Ext^i_{\mathcal{H}'}(\trivrep_{\mathcal{H}'},\trivrep_{\mathcal{H}'})\simeq \Ext^i_{\overline{\mathcal{H}}_\aff}(\trivrep_{\overline{\mathcal{H}}_\aff},\trivrep_{\overline{\mathcal{H}}_\aff})$.

Therefore it is sufficient to prove $\Ext^i_{\mathcal{H}}(\trivrep_G,\trivrep_G) = 0$ for $i > 0$ assuming $G$ is a split simply-connected semisimple group.
By \cite[Proposition~6.20]{MR3249689}, the projective dimension of $\trivrep_G$ is equal to the semisimple rank of $G$, namely $\#\Delta$.
Therefore $\Ext^{i + \#\Delta}_{\mathcal{H}}(\trivrep_G,\St_B(\trivrep_{B})) = 0$ for $i > 0$.
The left hand side is $\Ext^i_{\mathcal{H}}(\trivrep_G,\trivrep_G)$ by Theorem~\ref{thm:extension between steinbergs}.
\end{proof}

\subsection{Extension between extensions}
Let $P$ be a parabolic subgroup and $\sigma$ an $\mathcal{H}_P$-module which has the extension $e_G(\sigma)$ to $\mathcal{H}$.
In particular, $\Delta_P$ and $\Delta\setminus\Delta_P$ are orthogonal to each other.
Let $P_2$ be a parabolic subgroup corresponding to $\Delta\setminus\Delta_P$.
Let $J\subset \mathcal{H}$ be an ideal generated by $\{T_w^* - 1\mid w\in W_{\aff,P_2}(1)\}$.
Then $e_G(\sigma)(J) = 0$.
Hence for any module $\pi$ of $\mathcal{H}$, we have
\[
\Hom_\mathcal{H}(e(\sigma),\pi) = \Hom_{\mathcal{H}/J}(e(\sigma),\{v\in \pi\mid  vJ = 0\}).
\]
Note that $T_w^{P_2}\mapsto T_w$ defines the injection $\mathcal{H}_{\aff,P_2}\to \mathcal{H}$ since the restriction of $\ell$ on $W_{\aff,P_2}(1)$ is $\ell_{P_2}$.
Since any generator of $J$ is in $\mathcal{H}_{\aff,P_2}$, we have $\{v\in \pi\mid  vJ = 0\} = \{v\in \pi\mid  v(J\cap \mathcal{H}_{\aff,P_2}) = 0\}$.
Since the trivial representation $\trivrep_{\mathcal{H}_{\aff,P_2}}$ of $\mathcal{H}_{\aff,P_2}$ is isomorphic to $\mathcal{H}_{\aff,P_2}/(J\cap \mathcal{H}_{\aff,P_2})$, we get
\[
\{v\in \pi\mid  v(J\cap \mathcal{H}_{\aff,P_2}) = 0\}
=
\Hom_{\mathcal{H}_{\aff,P_2}}(\trivrep_{\mathcal{H}_{\aff,P_2}},\pi).
\]
Hence we get
\[
\Hom_\mathcal{H}(e_G(\sigma),\pi)
=
\Hom_{\mathcal{H}/J}(e_G(\sigma),\Hom_{\mathcal{H}_{\aff,P_2}}(\trivrep_{\mathcal{H}_{\aff,P_2}},\pi)).
\]
This isomorphism can be generalized as 
\[
\Hom_\mathcal{H}(\pi_1,\pi)
=
\Hom_{\mathcal{H}/J}(\pi_1,\Hom_{\mathcal{H}_{\aff,P_2}}(\trivrep_{\mathcal{H}_{\aff,P_2}},\pi))
\]
for any $\mathcal{H}/J$-module $\pi_1$.
In particular, $\pi\mapsto \Hom_{\mathcal{H}_{\aff,P_2}}(\trivrep_{\mathcal{H}_{\aff,P_2}},\pi)$ from the category of $\mathcal{H}$-modules to the category of $\mathcal{H}/J$-modules preserves injective modules.
Hence we have a spectral sequence
\[
\Ext^i_{\mathcal{H}/J}(e_G(\sigma),\Ext^j_{\mathcal{H}_{\aff,P_2}}(\trivrep_{\mathcal{H}_{\aff,P_2}},\pi))\Rightarrow \Ext^{i + j}_\mathcal{H}(e(\sigma),\pi).
\]
Now let $\sigma_1,\sigma_2$ be $\mathcal{H}_P$-modules such that both have the extensions $e_G(\sigma_1),e_G(\sigma_2)$ to $\mathcal{H}$.
Since $e_G(\sigma_2)|_{\mathcal{H}_{\aff,P_2}}$ is a direct sum of the trivial representations, we have
\[
\Ext^j_{\mathcal{H}_{\aff,P_2}}(\trivrep_{\mathcal{H}_{\aff,P_2}},e(\sigma_2)) = 0
\]
for $j > 0$ by Corollary~\ref{cor:vanishing of ext between trivials}.
Hence
\[
\Ext^i_\mathcal{H}(e_G(\sigma_1),e_G(\sigma_2))\simeq \Ext^i_{\mathcal{H}/J}(e_G(\sigma_1),e_G(\sigma_2)).
\]
\begin{lem}[{\cite[Proposition~3.5]{arXiv:1703.10384}}]
Let $I$ be the ideal of $\mathcal{H}_P$ generated by $\{T^P_\lambda - 1\mid \lambda\in \Lambda(1)\cap W_{\aff,P_2}(1)\}$.
Then we have $\mathcal{H}/J\simeq \mathcal{H}_P/I$.
\end{lem}
Therefore we get
\[
\Ext^i_\mathcal{H}(e_G(\sigma_1),e_G(\sigma_2))\simeq \Ext^i_{\mathcal{H}_P/I}(\sigma_1,\sigma_2).
\]

\begin{prop}\label{prop:quotient of Hecke algebra is abstract Hecke algebra}
Set $W'_\aff = W_{\aff,P}$, $S'_\aff = S_{\aff,P}$, $W' = W_P/(\Lambda\cap W_{\aff,P_2})$, $\Omega' = \Omega_P/(\Lambda\cap W_{\aff,P_2})$, $W'(1) = W_P(1)/(\Lambda(1)\cap W_{\aff,P_2}(1))$, $Z'_\kappa = Z_\kappa/(Z_\kappa\cap W_{\aff,P_2}(1))$.
Then $(W'_\aff,S'_\aff,\Omega',W',W'(1),Z'_\kappa)$ satisfies the condition of subsection~\ref{subsec:Prop-p-Iwahori Hecke algebra} and the attached algebra is $\mathcal{H}/I$.
Moreover $\Omega'$ is commutative.
\end{prop}
\begin{proof}
Since $\Delta = \Delta_P\cup \Delta_{P_2}$ is the orthogonal decomposition, we have $W_\aff = W_{\aff,P}\times W_{\aff,P_2}$ and $S_\aff = S_{\aff,P}\cup S_{\aff,P_2}$.
The pair $(W'_\aff,S'_\aff) = (W_{\aff,P},S_{\aff,P})$ is a Coxeter system and $\Omega_P$ acts on it.
Since $W_{\aff,P_2}$ commutes with $W_{\aff,P}$, this gives the action of $\Omega'$ on $(W'_\aff,S'_\aff)$.
We have $W'_\aff\subset W_P$ and since $W_{\aff,P}\cap W_{\aff,P_2}$ is trivial, we have the embedding $W'_\aff\subset W'$.
We also have $\Omega_P\subset W'$.
Since $W_P = W_{\aff,P}\Omega_P$, we have $W' = W'_\aff\Omega'$.
Since $W_{\aff,P}\cap \Omega_P = \{1\}$, we have $W'_\aff\cap \Omega' = \{1\}$ in $W'$.
Hence $W' = W'_\aff\rtimes\Omega'$.
Since $Z_\kappa$ is finite and commutative, $Z'_\kappa$ is also a finite commutative group.
The existence of the exact sequence
\[
1\to Z'_\kappa\to W'(1)\to W'\to 1
\]
is obvious.
Note that the length function $\ell'\colon W'(1)\to \Z_{\ge 0}$ is given by $\ell_P\colon W_P'(1)\to \Z_{\ge 0}$ since $S'_\aff = S_{\aff,P}$ and $\Omega'$ is the image of $\Omega_P$.

We put $q'_s = q_s$ for $s\in S_{\aff,P}$. (Note that $q_s = q_{s,P}$ since $\Delta = \Delta_P\cup \Delta_{P_2}$ is an orthogonal decomposition.)
For $s\in \Refs{W'(1)}$, take its lift $\widetilde{s}\in \Refs{W_P(1)}$ and let $c'_s$ be the image of $c_{\widetilde{s}}$ in $C[Z'_\kappa]$.
We prove that this is well-defined.
Let $\widetilde{s}'$ be another lift and take $\lambda\in \Lambda(1)\cap W_{\aff,P_2}(1)$ such that $\widetilde{s}' = \widetilde{s}\lambda$.
The image of $\widetilde{s}$ in $W$ is in $\Refs{W_P}\subset W_{P,\aff}$ since $S_{P,\aff}\subset W_{P,\aff}$ and $W_{P,\aff}$ is normal. (Recall that a reflection is an element which is conjugate to a simple reflection.)
Let $\overline{\lambda}$ be the image of $\lambda$ in $\Lambda$.
Since $\widetilde{s},\widetilde{s}'\in W_{P,\aff}(1)$, we have $\overline{\lambda}\in \Lambda\cap W_{\aff,P_2}\cap W_{\aff,P} = \{1\}$.
Hence $\lambda\in Z_\kappa$.
Since $\lambda\in W_{\aff,P_2}(1)$, we have $\lambda\in Z_\kappa\cap W_{\aff,P_2}(1)$.
Hence the image of $c_{\widetilde{s}'} = c_{\widetilde{s}\lambda} = c_{\widetilde{s}}\lambda$ is the same as that of $c_{\widetilde{s}}$ in $C[Z'_\kappa]$.

We get the parameter $(q',c')$ and let $\mathcal{H}' = \bigoplus_{w\in W'(1)}T'_w$ be the attached algebra.
Consider the linear map $\Phi\colon \mathcal{H}_P\to \mathcal{H}'$ defined by $T^P_w\mapsto T'_{\overline{w}}$ where $w\in W_P(1)$ and $\overline{w}\in W'(1)$ is the image of $w$.

First we prove that the map $\Phi$ preserves the relations.
Let $s\in W_P(1)$ be a lift of an affine simple reflection in $S_{\aff,P}$.
Then we have $(T^P_s)^2 = q_sT^P_{s^2} + c_sT^P_s$.
Let $\overline{s}$ be the image of $s$ in $W'(1)$.
Then we have $(T'_{\overline{s}})^2 = q'_{\overline{s}}T'_{\overline{s}^2} + c'_{\overline{s}}T'_{\overline{s}}$.
The definition of $(q',c')$ says $q'_{\overline{s}} = q_s$ and $\Phi(c_s) = c'_{\overline{s}}$.
Hence $\Phi$ preserves the quadratic relations.
The compatibility between $\ell_P$ and $\ell'$ implies that $\Phi$ preserves the braid relations.

Obviously $\Phi$ is surjective.
We prove that $\Ker\Phi = I$.
Clearly we have $I\subset \Ker\Phi$.
Let $\sum_{w\in W_P(1)}c_wT_w\in \Ker\Phi$ where $c_w\in C$.
Fix a section $x$ of $W_P(1)\to W'(1)$.
Then we have $\sum_{w\in W_P(1)}c_wT^P_w = \sum_{w\in W'(1)}\sum_{\lambda\in \Lambda(1)\cap W_{\aff,P_2}(1)}c_{x(w)\lambda}T^P_{x(w)\lambda}$.
Hence 
\[
0 = \Phi\left(\sum_{w\in W_P(1)}c_wT^P_w\right) = \sum_{w\in W'(1)}\left(\sum_{\lambda\in \Lambda(1)\cap W_{\aff,P_2}(1)}c_{x(w)\lambda}\right)T'_{w}.
\]
Therefore fore each $w\in W'(1)$ we have $\sum_{\lambda\in \Lambda(1)\cap W_{\aff,P_2}(1)}c_{x(w)\lambda} = 0$.
Hence
\begin{align*}
& \sum_{w\in W_P(1)}c_wT_w^P\\
& = \sum_{w\in W'(1)}\sum_{\lambda\in \Lambda(1)\cap W_{\aff,P_2}(1)}c_{x(w)\lambda}T_{x(w)\lambda}^P\\
& = \sum_{w\in W'(1)}\left(\sum_{\lambda\in \Lambda(1)\cap W_{\aff,P_2}(1)}c_{x(w)\lambda}T^P_{x(w)\lambda} - \sum_{\lambda\in \Lambda(1)\cap W_{\aff,P_2}(1)}c_{x(w)\lambda}T^P_{x(w)}\right)\\
& = \sum_{w\in W'(1)}\left(\sum_{\lambda\in \Lambda(1)\cap W_{\aff,P_2}(1)}c_{x(w)\lambda}T^P_{x(w)}(T^P_\lambda - 1)\right)\in I.
\end{align*}

Finally, $\Omega'$ is commutative since $\Omega_P$ is commutative.
\end{proof}
\begin{rem}
The data does not come from a reductive group in general.
\end{rem}

\subsection{Example}
Let $G = \mathrm{PGL}_2$.
We have $\Lambda(1) \simeq F^\times / (1 + (\varpi))\simeq \Z\times \kappa^\times$.
Consider $\widetilde{G} = \mathrm{SL}_2$.
Then $G'$ is the image of $\widetilde{G}\to G$~\cite[II.4 Proposition]{MR3600042}.
By this description, we have $\Lambda(1)\cap W_\aff(1) =  \{\lambda^2\mid \lambda\in \Lambda(1)\}$.
Therefore, with the notation in Proposition~\ref{prop:quotient of Hecke algebra is abstract Hecke algebra}, we have $W'_\aff = \{1\}$, $S'_{\aff} = \emptyset$, $W'(1) = \Lambda(1)/\{\lambda^2\mid \lambda\in \Lambda(1)\}$, $Z_\kappa = Z_\kappa/\{t^2\mid t\in Z_\kappa\}$.
We have $\mathcal{H}_B/I = C[W'(1)]$.

Consider the trivial module $\trivrep_G$.
Then we have $\trivrep_G = e_G(\trivrep_B)$ and we have
\[
\Ext^i_\mathcal{H}(\trivrep_G,\trivrep_G)
\simeq
\Ext^i_{\mathcal{H}_B/I}(\trivrep_B,\trivrep_B)
\simeq
\Ext^i_{C[W'(1)]}(\trivrep_B,\trivrep_B)
=
H^i(W'(1),C).
\]
Here $C$ is the trivial $W'(1)$-module.
Since the group $W'(1)$ is a $2$-group, this cohomology if zero if the characteristic of $C$ is not $2$.
However if the characteristic of $C$ is $2$, since $W'(1)\simeq \Z/2\Z$ ($p = 2$) or $(\Z/2\Z)^{\oplus 2}$ ($p\ne 2$), $H^i(W'(1),C)\ne 0$ if $i$ is even.
Therefore we have infinitely many $i$ with $\Ext^i_\mathcal{H}(\trivrep_G,\trivrep_G)\ne 0$.
This recovers Koziol's example~\cite[Example~6.2]{arXiv:1512.00247}.

\subsection{Summary}
Now we get a reduction.
The $\Ext^1$ between simple modules is equal to $\Ext^{1 - r}$ between supersingular simple modules for some $r\ge 0$ or zero.
In particular, if $r\ge 2$, then $\Ext^1$ between simple modules is zero.
If $r = 1$, then $\Ext^{1 - r} = \Hom$, so it is zero or one-dimensional.
If $r = 0$, we have to calculate $\Ext^1$ between supersingular simple modules.
Therefore, the only remaining task is to calculate $\Ext^1$ between supersingular simple modules.

\section{$\Ext^1$ between supersingular modules}\label{sec:Ext1 between supersingulars}
In this section, we fix a data $(W_\aff,S_\aff,\Omega,W,W(1),Z_\kappa)$ and let $\mathcal{H}$ be the algebra attached to this data.
We do not assume that this data comes from the data.
We also assume:
\begin{itemize}
\item our parameter $q_s$ is zero.
\item $\#Z_\kappa$ is prime to $p$.
\end{itemize}
As in subsection~\ref{subsec:supersingulars}, let $W^\aff(1)$ be the inverse image of $W_\aff$ in $W(1)$ and put $\mathcal{H}^\aff = \bigoplus_{w\in W^\aff(1)}CT_w$.

For a character $\chi$ of $Z_\kappa$ and $w\in W$, we define $(w\chi)(t) = \chi(\widetilde{w}^{-1}t\widetilde{w})$ where $\widetilde{w}\in W(1)$ is a lift of $W$.
Since $Z_\kappa$ is commutative, this does not depend on a lift $\widetilde{w}$ and defines a character $w\chi$ of $Z_\kappa$.
For a character $\Xi$ of $\mathcal{H}^\aff$ and $\omega\in \Omega(1)$, we write $\Xi\omega$ for the character $T_w\mapsto \Xi(T_{\omega w\omega^{-1}})$ for $w\in W^\aff(1)$.
Since $\Xi\omega$ only depends on the image $\overline{\omega}$ of $\omega$ in $\Omega$, we also write $\Xi\overline{\omega}$.

Note that since $s\cdot c_{\widetilde{s}} = c_{\widetilde{s}}$ for $s\in S_\aff$ with a lift $\widetilde{s}$ by the conditions of the parameter $c$, we have $(s\chi)(c_{\widetilde{s}}) = \chi(c_{\widetilde{s}})$.

\subsection{$\Ext^1$ for $\mathcal{H}^\aff$}
Let $\chi,\chi'$ be characters of $Z_\kappa$ and $J\subset S_{\aff,\chi},J'\subset S_{\aff,\chi'}$ subsets.
Then we have characters $\Xi = \Xi_{J,\chi}$, $\Xi' = \Xi_{J',\chi'}$ of $\mathcal{H}^\aff$.
We calculate $\Ext^1_{\mathcal{H}^\aff}(\Xi,\Xi')$.

To express the space of extensions, we need some notation.
For each $s\in S_\aff$, let $C_s$ be the set of functions $a$ on $\{\widetilde{s}\in W(1)\mid \widetilde{s}\mapsto s\in W\}$ such that $a(t\widetilde{s}) = \chi'(t)a(\widetilde{s})$, $a(\widetilde{s}t) = a(\widetilde{s})\chi(t)$ for any $t\in Z_\kappa$.
Then $C_s\ne 0$ if and only if $\chi' = s\chi$ and if $\chi' = s\chi$ then $\dim_C C_s = 1$.

Now we define some subsets of $S_\aff$.
First consider the sets 
\begin{align*}
A_1(\Xi,\Xi') & = \{s\in S_\aff\mid \Xi(T_{\widetilde{s}}) = \Xi'(T_{\widetilde{s}}) = 0\},\\
A_2(\Xi,\Xi') & = \{s\in S_\aff\mid \Xi(T_{\widetilde{s}}) \ne 0,\ \Xi'(T_{\widetilde{s}}) = 0\},\\
A_3(\Xi,\Xi') & = \{s\in S_\aff\mid \Xi(T_{\widetilde{s}}) = 0,\ \Xi'(T_{\widetilde{s}}) \ne 0\},\\
A_4(\Xi,\Xi') & = \{s\in S_\aff\mid \Xi(T_{\widetilde{s}}) \ne 0,\ \Xi'(T_{\widetilde{s}}) \ne 0\}.
\end{align*}
where $\widetilde{s}$ is a lift of $s$.
We define
\begin{align*}
S_2(\Xi,\Xi') & = A_2(\Xi,\Xi')\cup A_3(\Xi,\Xi')\\
S_1(\Xi,\Xi') & = \{s\in A_1(\Xi,\Xi')\setminus S_{\aff,\chi}\mid s\chi = \chi',\ \text{$(ss_1)^2 \ne 1$ for any $s_1\in S_2(\Xi,\Xi')$}\}.
\end{align*}

If $s\in S_{\aff,\chi}$ and $a\in C_s$, then $a(\widetilde{s})\chi(c_{\widetilde{s}})^{-1}\in C$ does not depend on a lift $\widetilde{s}$ of $s$.
We denote it by $a\chi(c_s)^{-1}$.
We also have that if $s\in S_{\aff,\chi'}$ then $a(\widetilde{s})\chi'(c_{\widetilde{s}})^{-1}$ does not depend on a lift $\widetilde{s}$.
We denote it by $a\chi'(c_s)^{-1}$.
If $a\ne 0$, then $\chi' = s \chi$.
Hence if $s\in S_{\aff,\chi}$ then $s\in S_{\aff,\chi'}$ and $a\chi(c_s)^{-1} = a\chi'(c_s)^{-1}$.

For the Hecke algebra attached to a finite Coxeter system, the following proposition is \cite[Theorem~5.1]{MR2134290} and we use a similar proof.
\begin{prop}\label{prop:extension between supersingulars, for H_aff}
Consider the subspace $E_2(\Xi,\Xi')$ of $\bigoplus_{s\in S_2(\Xi,\Xi')}C_s$ consisting $(a_s)$ such that
\begin{itemize}
\item If $s_1,s_2\in A_2(\Xi,\Xi')$, then $a_{s_1}\chi(c_{s_1})^{-1} = a_{s_2}\chi(c_{s_2})^{-1}$.
\item If $s_1,s_2\in A_3(\Xi,\Xi')$, then $a_{s_1}\chi'(c_{s_1})^{-1} = a_{s_2}\chi'(c_{s_2})^{-1}$.
\item If $s_1\in A_2(\Xi,\Xi')$, $s_2\in A_3(\Xi,\Xi')$ and $(s_1s_2)^2 = 1$, then $a_{s_1}\chi(c_{s_1})^{-1} + a_{s_2}\chi'(c_{s_2})^{-1} = 0$.
\end{itemize}
and put $E_1(\Xi,\Xi') = \bigoplus_{s\in S_1(\Xi,\Xi')}C_s$, $E(\Xi,\Xi') = E_1(\Xi,\Xi')\oplus E_2(\Xi,\Xi')$.
For $(a_s)\in E(\Xi,\Xi')$, consider the linear map $\mathcal{H}\to M_2(C)$ defined by
\[
T_{\widetilde{s}} \mapsto \begin{pmatrix}\Xi(T_{\widetilde{s}}) & 0\\ a_s(\widetilde{s}) & \Xi'(T_{\widetilde{s}})\end{pmatrix},
\]
where $a_s = 0$ if $s\notin S_1(\Xi,\Xi')\cup S_2(\Xi,\Xi')$.
Then this gives an extension of $\Xi$ by $\Xi'$ and it gives a surjective map $E(\Xi,\Xi')\to \Ext_{\mathcal{H}^\aff}^1(\Xi,\Xi')$.
The kernel is
\begin{equation}\label{eq:kernel of E->Ext1 (2)}
\left\{(a_s)\in E(\Xi,\Xi')\;\middle|\; 
\begin{array}{l}
a_{s_1}\chi(c_{s_1})^{-1} + a_{s_2}\chi'(c_{s_2})^{-1} = 0\\
\quad\quad\quad\quad (s_1\in A_2(\Xi,\Xi'),s_2\in A_3(\Xi,\Xi'))\\
a_s = 0\ (s\in S_1(\Xi,\Xi'))
\end{array}\right\}.
\end{equation}
\end{prop}
\begin{rem}
Let $V_2$ be a subspace of $\bigoplus_{s\in A_2(\Xi,\Xi')}C_s$ consisting $(a_s)$ such that $a_{s_1}\chi(c_{s_1})^{-1} = a_{s_2}\chi(c_{s_2})^{-1}$ for any $s_1,s_2\in A_2(\Xi,\Xi')$.
Then $\dim V_2\le 1$ and $V_2\ne 0$ if and only if $C_s\ne 0$ for any $s\in A_2(\Xi,\Xi')$, namely $s\chi = \chi'$ for any $s\in A_2(\Xi,\Xi')$.
Define $V_3$ by the similar way.
Then $\dim V_3\le 1$ and $V_3\ne 0$ if and only if $s\chi = \chi'$ for any $s\in A_3(\Xi,\Xi')$.
If there is no $s_1\in A_2(\Xi,\Xi')$ and $s_2\in A_3(\Xi,\Xi')$ such that $(s_1s_2)^2 = 1$, then $E_2(\Xi,\Xi') = V_2\oplus V_3$.
Otherwise $\dim E_2(\Xi,\Xi') = \max\{0,\dim V_2 + \dim V_3 - 1\}$.
\end{rem}

\begin{proof}
Let $M$ be an extension of $\Xi$ by $\Xi'$.
Since $\#Z_\kappa$ is prime to $p$, the representation of $Z_\kappa$ over $C$ is completely reducible.
Hence we can take a basis $e_1,e_2$ such that $T_te_1 = \chi(t)e_1$ and $T_te_2 = \chi'(t)e_2$.
With this basis, the action of $T_{\widetilde{s}}$ where $\widetilde{s}\in S_\aff(1)$ with the image $s\in S_\aff$ is described as
\[
T_{\widetilde{s}} = \begin{pmatrix}\Xi(T_{\widetilde{s}}) & 0\\ a_s(\widetilde{s}) & \Xi'(T_{\widetilde{s}})\end{pmatrix}.
\]
for some $a_s(\widetilde{s})\in C$.
The action of $T_t$ where $t\in Z_\kappa$ is given by
\[
\begin{pmatrix}
\chi(t) & 0\\0 & \chi'(t).
\end{pmatrix}
\]
Since $T_tT_{\widetilde{s}} = T_{t\widetilde{s}}$, we have
\[
\begin{pmatrix}
\chi(t) & 0\\0 & \chi'(t).
\end{pmatrix}
\begin{pmatrix}\Xi(T_{\widetilde{s}}) & 0\\ a_s(\widetilde{s}) & \Xi'(T_{\widetilde{s}})\end{pmatrix}
=
\begin{pmatrix}\Xi(T_{t\widetilde{s}}) & 0\\ a_s(t\widetilde{s}) & \Xi'(T_{t\widetilde{s}})\end{pmatrix}.
\]
Hence $a_s(t\widetilde{s}) = \chi'(t)a_s(\widetilde{s})$.
Similarly we have $a_s(\widetilde{s}t) = a_s(\widetilde{s})\chi(t)$.
Hence $a_s\in C_s$.

Now we check the conditions that the map defines an action of $\mathcal{H}^\aff$.
Since we have
\[
\begin{pmatrix}
\Xi(T_{\widetilde{s}}) & 0\\ a_s(\widetilde{s}) & \Xi'(T_{\widetilde{s}})
\end{pmatrix}^2
=
\begin{pmatrix}
\Xi(T_{\widetilde{s}})^2 & 0\\ a_s(\widetilde{s})(\Xi(T_{\widetilde{s}}) + \Xi'(T_{\widetilde{s}})) & \Xi'(T_{\widetilde{s}})^2
\end{pmatrix},
\]
this satisfies the quadratic relation $T_{\widetilde{s}}^2 = T_{\widetilde{s}}c_{\widetilde{s}}$ if and only if
\begin{equation}\label{eq:extension satisfies quadratic relation}
a_s(\widetilde{s})(\Xi(T_{\widetilde{s}}) + \Xi'(T_{\widetilde{s}})) = a_s(\widetilde{s})\chi(c_{\widetilde{s}}).
\end{equation}

If $s\in A_1(\Xi,\Xi')$, then $a_s(\widetilde{s}) = 0$ or $\chi(c_{\widetilde{s}}) = 0$, namely $a_s = 0$ or $s\notin S_{\aff,\chi}$.

If $s\in A_2(\Xi,\Xi')$, then $a_s =0$ or $\Xi(T_{\widetilde{s}}) = \chi(c_{\widetilde{s}})$.
Since $\Xi(T_{\widetilde{s}}) \ne 0$, we always have $\Xi(T_{\widetilde{s}}) = \chi(c_{\widetilde{s}})$.
Hence \eqref{eq:extension satisfies quadratic relation} is always satisfied.

If $s\in A_3(\Xi,\Xi')$, then $a_s =0$ or $\Xi'(T_{\widetilde{s}}) = \chi(c_{\widetilde{s}})$.
Note that if $a_s\ne 0$ then $s \chi = \chi'$, hence $S_{\aff,\chi} = S_{\aff,\chi'}$ and $\chi(c_{\widetilde{s}}) = \chi'(c_{\widetilde{s}})$.
Therefore under $a_s\ne 0$, we have $\Xi'(T_{\widetilde{s}}) = \chi(c_{\widetilde{s}})$ if and only if $\Xi'(T_{\widetilde{s}}) = \chi'(c_{\widetilde{s}})$.
This always hold since $\Xi'(T_{\widetilde{s}}) \ne 0$.
Hence \eqref{eq:extension satisfies quadratic relation} is always satisfied.

If $s\in A_4(\Xi,\Xi')$, then we have $\Xi(T_{\widetilde{s}}) = \chi(c_{\widetilde{s}})$.
Hence we have $a_s(\widetilde{s})\Xi'(T_{\widetilde{s}}) =0$.
Therefore we have $a_s = 0$ since $\Xi'(T_{\widetilde{s}}) \ne 0$.

Consequently the quadratic relation holds if and only if $a_s = 0$ or $s\in (A_1(\Xi,\Xi')\setminus S_{\aff,\chi})\cup A_2(\Xi,\Xi')\cup A_3(\Xi,\Xi')$.
The action of $T_{\widetilde{s}}$ is given by one of the following matrix:
\begin{equation}\label{eq:type of matrices}
\begin{pmatrix}
0 & 0\\ a_s(\widetilde{s}) & 0
\end{pmatrix},
\begin{pmatrix}
\chi(c_{\widetilde{s}}) & 0\\ a_s(\widetilde{s}) & 0
\end{pmatrix},
\begin{pmatrix}
0 & 0\\ a_s(\widetilde{s}) & \chi'(c_{\widetilde{s}})
\end{pmatrix},
\begin{pmatrix}
\chi(c_{\widetilde{s}}) & 0\\0 & \chi'(c_{\widetilde{s}})
\end{pmatrix}.
\end{equation}
Here each $\chi(c_{\widetilde{s}})$ and $\chi'(c_{\widetilde{s}})$ is not zero and in the first matrix, we assume that $s\notin S_{\aff,\chi}$ if $a_s\ne 0$.

Now we check the braid relations.
Let $s_1,s_2\in S_\aff$ and $\widetilde{s}_1,\widetilde{s}_2$ their lifts.
We consider a braid relation $s_1s_2\dotsm  = s_2s_1\dotsm $.
It is easy to see that the action satisfies the braid relation for some lifts $\widetilde{s}_1,\widetilde{s}_2$ if and only if it is satisfied for any lifts $\widetilde{s}_1,\widetilde{s}_2$.
Take $\widetilde{s}_1,\widetilde{s}_2$ such that $\widetilde{s}_1\widetilde{s}_2\dotsm = \widetilde{s}_2\widetilde{s}_1\dotsm $.
It is easy to see that if $s_1\in A_4(\Xi,\Xi')$ or $s_2\in A_4(\Xi,\Xi')$, then the braid relations hold automatically.
So we assume that $s_1,s_2\in A_1(\Xi,\Xi')\cup A_2(\Xi,\Xi')\cup A_3(\Xi,\Xi')$.

Assume that $s_1\in A_1(\Xi,\Xi')$.
We have
\begin{gather*}
\begin{pmatrix}0 & 0\\ a_{s_1}(\widetilde{s}_1) & 0\end{pmatrix}
\begin{pmatrix}\Xi(T_{\widetilde{s}_2}) & 0\\ a_{s_2}(\widetilde{s}_2) & \Xi'(T_{\widetilde{s}_2})\end{pmatrix}=
\begin{pmatrix}
0 & 0\\ a_{s_1}(\widetilde{s}_1)\Xi(T_{\widetilde{s}_2}) & 0
\end{pmatrix},\\
\begin{pmatrix}0 & 0\\ a_{s_1}(\widetilde{s}_1) & 0\end{pmatrix}
\begin{pmatrix}\Xi(T_{\widetilde{s}_2}) & 0\\ a_{s_2}(\widetilde{s}_2) & \Xi'(T_{\widetilde{s}_2})\end{pmatrix}\begin{pmatrix}0 & 0\\ a_{s_1}(\widetilde{s}_1) & 0\end{pmatrix}
=
\begin{pmatrix}
0 & 0\\ 0 & 0
\end{pmatrix},\\
\begin{pmatrix}\Xi(T_{\widetilde{s}_2}) & 0\\ a_{s_2}(\widetilde{s}_2) & \Xi'(T_{\widetilde{s}_2})\end{pmatrix}\begin{pmatrix}0 & 0\\ a_{s_1}(\widetilde{s}_1) & 0\end{pmatrix}
=
\begin{pmatrix}
0 & 0 \\ \Xi'(T_{\widetilde{s}_2})a_{s_1}(\widetilde{s}_1) & 0
\end{pmatrix},\\
\begin{pmatrix}\Xi(T_{\widetilde{s}_2}) & 0\\ a_{s_2}(\widetilde{s}_2) & \Xi'(T_{\widetilde{s}_2})\end{pmatrix}\begin{pmatrix}0 & 0\\ a_{s_1}(\widetilde{s}_1) & 0\end{pmatrix}
\begin{pmatrix}\Xi(T_{\widetilde{s}_2}) & 0\\ a_{s_2}(\widetilde{s}_2) & \Xi'(T_{\widetilde{s}_2})\end{pmatrix}=
\begin{pmatrix}
0 & 0 \\ \Xi'(T_{\widetilde{s}_2})\Xi(T_{\widetilde{s}_2})a_{s_1}(\widetilde{s}_1) & 0
\end{pmatrix},\\
\begin{pmatrix}\Xi(T_{\widetilde{s}_2}) & 0\\ a_{s_2}(\widetilde{s}_2) & \Xi'(T_{\widetilde{s}_2})\end{pmatrix}\begin{pmatrix}0 & 0\\ a_{s_1}(\widetilde{s}_1) & 0\end{pmatrix}
\begin{pmatrix}\Xi(T_{\widetilde{s}_2}) & 0\\ a_{s_2}(\widetilde{s}_2) & \Xi'(T_{\widetilde{s}_2})\end{pmatrix}\begin{pmatrix}0 & 0\\ a_{s_1}(\widetilde{s}_1) & 0\end{pmatrix}=
\begin{pmatrix}
0 & 0\\ 0 & 0
\end{pmatrix}.
\end{gather*}
Hence the braid relation is satisfied if and only if
\begin{itemize}
\item $a_{s_1} = 0$
\item or $\Xi(T_{\widetilde{s}_2}) = \Xi'(T_{\widetilde{s}_2})$ and the order of $s_1s_2$ is $2$
\item or $\Xi(T_{\widetilde{s}_2})\Xi'(T_{\widetilde{s}_2}) = 0$ and the order of $s_1s_2$ is $3$
\item or the order of $s_1s_2$ is greater than $3$.
\end{itemize}
If $s_2\in A_1(\Xi,\Xi')$, then the condition always holds.
If $s_2\in A_2(\Xi,\Xi')\cup A_3(\Xi,\Xi')$, then the condition holds if and only if $a_{s_1} = 0$ or the order of $s_1s_2$ is not $2$, namely $(s_1s_2)^2 \ne 1$.

Replacing $s_1$ with $s_2$, if $s_2\in A_1(\Xi,\Xi')$, we have the similar condition.

Assume that $s_1,s_2\in A_2(\Xi,\Xi')$.
We have
\[
\begin{pmatrix}\chi(c_{\widetilde{s}_1}) & 0\\ a_{s_1}(\widetilde{s}_1) & 0\end{pmatrix}
\begin{pmatrix}\chi(c_{\widetilde{s}_2}) & 0\\ a_{s_2}(\widetilde{s}_2) & 0\end{pmatrix}
=
\begin{pmatrix}\chi(c_{\widetilde{s}_1})\chi(c_{\widetilde{s}_2}) & 0\\ a_{s_1}(\widetilde{s}_1)\chi(c_{\widetilde{s}_2}) & 0\end{pmatrix}
=
\begin{pmatrix}\chi(c_{\widetilde{s}_1}) & 0\\ a_{s_1}(\widetilde{s}_1) & 0\end{pmatrix}\chi(c_{\widetilde{s}_2}).
\]
By this calculation, the braid relation is satisfied if and only if $a_{s_1}(\widetilde{s}_1)\chi(c_{\widetilde{s}_2})\dotsm = a_{s_2}(\widetilde{s}_2)\chi(c_{\widetilde{s}_1})\dotsm$.
By \cite[Proposition~4.13 (6)]{MR3484112}, we have $c_{\widetilde{s}_1}c_{\widetilde{s}_2}\dotsm = c_{\widetilde{s}_2}c_{\widetilde{s}_1}\dotsm$.
Hence the braid relation is satisfied if and only if $a_{s_1}(\widetilde{s}_1)\chi(c_{\widetilde{s}_1})^{-1} = a_{s_2}(\widetilde{s}_2)\chi(c_{\widetilde{s}_2})^{-1}$, namely $a_{s_1}\chi(c_{s_1})^{-1} = a_{s_2}\chi(c_{s_2})^{-1}$.
By a similar calculation, if $s_1,s_2\in A_3(\Xi,\Xi')$, then the braid relation is satisfied if and only if $a_{s_1}\chi'(c_{s_1})^{-1} = a_{s_2}\chi'(c_{s_2})^{-1}$.

Finally we assume that $s_1\in A_2(\Xi,\Xi')$ and $s_2\in A_3(\Xi,\Xi')$.
We have
\begin{gather*}
\begin{pmatrix}\chi(c_{\widetilde{s}_1}) & 0\\ a_{s_1}(\widetilde{s}_1) & 0\end{pmatrix}
\begin{pmatrix}0 & 0\\ a_{s_2}(\widetilde{s}_2) & \chi'(c_{\widetilde{s}_2})\end{pmatrix}
=
\begin{pmatrix}
0 & 0\\ 0 & 0
\end{pmatrix},\\
\begin{pmatrix}0 & 0\\ a_{s_2}(\widetilde{s}_2) & \chi'(c_{\widetilde{s}_2})\end{pmatrix}
\begin{pmatrix}\chi(c_{\widetilde{s}_1}) & 0\\ a_{s_1}(\widetilde{s}_1) & 0\end{pmatrix}
=
\begin{pmatrix}
0 & 0\\ a_{s_2}(\widetilde{s}_2)\chi(c_{\widetilde{s}_1}) + a_{s_1}(\widetilde{s}_1)\chi'(c_{\widetilde{s}_2}) & 0
\end{pmatrix},\\
\begin{pmatrix}0 & 0\\ a_{s_2}(\widetilde{s}_2) & \chi'(c_{\widetilde{s}_2})\end{pmatrix}
\begin{pmatrix}\chi(c_{\widetilde{s}_1}) & 0\\ a_{s_1}(\widetilde{s}_1) & 0\end{pmatrix}
\begin{pmatrix}0 & 0\\ a_{s_2}(\widetilde{s}_2) & \chi'(c_{\widetilde{s}_2})\end{pmatrix}
=
\begin{pmatrix}
0 & 0\\ 0 & 0
\end{pmatrix}.
\end{gather*}
Hence the braid relation is satisfied if and only if
\begin{itemize}
\item $a_{s_2}(\widetilde{s}_2)\chi(c_{\widetilde{s}_1}) + a_{s_1}(\widetilde{s}_1)\chi'(c_{\widetilde{s}_2}) = 0$ and the order of $s_1s_2$ is $2$.
\item or the order of $s_1s_2$ is greater than $2$.
\end{itemize}
We notice that $a_{s_2}(\widetilde{s}_2)\chi(c_{\widetilde{s}_1}) + a_{s_1}(\widetilde{s}_1)\chi'(c_{\widetilde{s}_2}) = 0$ if and only if $a_{s_1}\chi(c_{s_1})^{-1} + a_{s_2}\chi'(c_{s_2})^{-1} = 0$.

We get the following table which shows the condition for the braid relation:
\begin{center}
\begin{tabular}{c|c|c|c|c}
\diagbox{$s_2$}{$s_1$}& $A_1$ & $A_2$ & $A_3$ & $A_4$\\\hline
$A_1$ & always & \multicolumn{2}{c|}{$a_{s_2} = 0$ or $(s_1s_2)^2\ne 1$} & always \\\hline
\raisebox{1.3em}{$A_2$}  & \multirow{2}{*}{\shortstack{$a_{s_1}  = 0$\\ or $(s_1s_2)^2\ne 1$}} & \raisebox{.7em}{\shortstack{$a_{s_1}\chi(c_{s_1})^{-1}$\\\quad $= a_{s_2}\chi(c_{s_2})^{-1}$}} & \shortstack{$a_{s_1}\chi'(c_{s_1})^{-1} +$\\ $a_{s_2}\chi(c_{s_2})^{-1} = 0$\\ or $(s_1s_2)^2 \ne 1$} & \raisebox{1.4em}{always}\\\cline{1-1}\cline{3-5}
\raisebox{1.4em}{$A_3$}  &  & \shortstack{$a_{s_1}\chi(c_{s_1})^{-1} +$\\ $a_{s_2}\chi'(c_{s_2})^{-1} = 0$\\ or $(s_1s_2)^2 \ne 1$} & \raisebox{.7em}{\shortstack{$a_{s_1}\chi'(c_{s_1})^{-1}$\\\quad$ = a_{s_2}\chi'(c_{s_2})^{-1}$}}& \raisebox{1.4em}{always}\\\hline
$A_4$ & always & always & always & always
\end{tabular}
\end{center}

Now we assume that $(a_s)\in \bigoplus_{s \in S_\aff}C_s$ defines an action of $\mathcal{H}$.
First recall that $C_s\ne 0$ if and only if $s\chi = \chi'$.
Hence $a_s\ne 0$ implies $s\chi = \chi'$.
Since the quadratic relations hold, if $a_s\ne 0$ then $s\in (A_1(\Xi,\Xi')\setminus S_{\aff,\chi})\cup S_2(\Xi,\Xi')$.
If $s\in A_1(\Xi,\Xi')\setminus S_{\aff,\chi}$ and $(s s_1)^2 = 1$ for some $s_1\in S_2(\Xi,\Xi')$, then the table says that $a_s = 0$.
Therefore if $a_s\ne 0$ then $s\in S_1(\Xi,\Xi')\cup S_2(\Xi,\Xi')$.
Hence again by the table, $(a_s)$ belongs to $E(\Xi,\Xi')$.

Conversely, if $(a_s)\in E(\Xi,\Xi')$, then $a_s\ne 0$ implies $s\in S_1(\Xi,\Xi')\cup S_2(\Xi,\Xi')\subset (A_1(\Xi,\Xi')\setminus S_{\aff,\chi})\cup S_2(\Xi,\Xi')$.
Hence each $T_{\widetilde{s}}$ satisfies the quadratic relation.
Let $s_1,s_2\in S_{\aff}$.
If $s_1\in A_1(\Xi,\Xi')$ and $s_2\in A_2(\Xi,\Xi')\cup A_3(\Xi,\Xi')$, then the definition of $S_1(\Xi,\Xi')$ says that $(s_1s_2)^2 \ne 1$ or $a_{s_1} = 0$.
Then by the table the braid relation for $s_1,s_2$ holds.
For other cases, the condition on $E_2(\Xi,\Xi')$ and the table implies that the braid relation holds too.

Therefore the map $E(\Xi,\Xi')\to \Ext^1_{\mathcal{H}^\aff}(\Xi,\Xi')$ is well-defined and surjective.

Assume that the extension given by $(a_s)$ splits, namely each matrices in \eqref{eq:type of matrices} are simultaneous diagonalizable.
If $s\in A_1(\Xi,\Xi')$, then the matrix corresponding to $s$ is diagonalizable if and only if $a_s = 0$.
Let $s_1,s_2\in A_2(\Xi,\Xi')$.
Then by $a_{s_1}\chi(c_{s_1})^{-1} = a_{s_2}\chi(c_{s_2})^{-1}$, the matrices corresponding to $s_1,s_2$ commutes with each other.
Hence these matrices are simultaneous diagonalizable.
Similarly, matrices corresponding to $A_3(\Xi,\Xi')$ are simultaneous diagonalizable.

If $s_1\in A_2(\Xi,\Xi')$ and $s_2\in A_3(\Xi,\Xi')$, then the corresponding matrices are
\[
\begin{pmatrix}
\chi(c_{\widetilde{s}_1}) & 0\\ a_s(\widetilde{s}_1) & 0
\end{pmatrix},
\begin{pmatrix}
0 & 0\\ a_s(\widetilde{s}_2) & \chi'(c_{\widetilde{s}_2})
\end{pmatrix}.
\]
and these commute with each other if and only if $a_{s_1}\chi(c_{s_1})^{-1} + a_{s_2}\chi(c_{s_2})^{-1} = 0$.
Hence the kernel is \eqref{eq:kernel of E->Ext1 (2)}.
\end{proof}

\begin{rem}
Let $\omega\in \Omega(1)_{\Xi}\cap \Omega(1)_{\Xi'}$.
Then $\omega$ acts on $\Ext^1_{\mathcal{H}^\aff}(\Xi,\Xi')$ and by this action $\Ext^1_{\mathcal{H}^\aff}(\Xi,\Xi')$ is a right $\Omega(1)_{\Xi}\cap \Omega(1)_{\Xi'}$-module.
We also have the action of $\omega$ on $E(\Xi,\Xi')$ as follows:
Let $s\in S_\aff$ and put $s_1 = \omega^{-1} s\omega\in S_\aff$.
Then $a\mapsto (\widetilde{s}_1\mapsto a(\omega\widetilde{s}_1\omega^{-1}))$ gives an isomorphism $C_s\simeq C_{s_1}$.
We denote this map by $a\mapsto a\cdot \omega$.
Then the action is given by $(a_s)\mapsto (a_{\omega^{-1}s\omega}\cdot \omega)$.
This action commutes with the action of $\omega$ on $\Ext^1_{\mathcal{H}^\aff}(\Xi,\Xi')$.
\end{rem}

\subsection{Semi-direct product}
The argument in this subsection is general.
Let $A$ be a $C$-algebra and $\Gamma$ a group acting on $A$.
We assume that a finite commutative normal subgroup $\Gamma'\subset \Gamma$ and an embedding $C[\Gamma']\hookrightarrow A$ are given.
Here we assume that for $\gamma'\in \Gamma'$, the action of $\gamma'$ on $A$ as an element in $\Gamma$ is given by $a\mapsto \gamma'a(\gamma')^{-1}$.
We put $B = C[\Gamma]\otimes_{C[\Gamma']}A$ and define a multiplication by $(\gamma_1\otimes a_1)(\gamma_2\otimes a_2) = \gamma_1\gamma_2\otimes (\gamma_2^{-1}\cdot a_1)a_2$ for $a_1,a_2\in A$ and $\gamma_1,\gamma_2\in \Gamma$.
Of course, the example in our mind is $A = \mathcal{H}^\aff$, $\Gamma = \Omega(1)$ and $\Gamma' = Z_\kappa$.
We have $B = \mathcal{H}$.

Let $M_1,M_2$ be right $B$-modules.
Then $\Hom_A(M_1,M_2)$ has the structure of a $\Gamma$-module defined by $(f\gamma)(m) = f(m\gamma^{-1})\gamma$.
This action factors through $\Gamma \to \Gamma/\Gamma'$ and we have $\Hom_B(M_1,M_2) = \Hom_A(M_1,M_2)^{\Gamma/\Gamma'}$.
Let $N$ be a $\Gamma/\Gamma'$-module and $\varphi\in \Hom_{\Gamma/\Gamma'}(N,\Hom_A(M_1,M_2))$.
Set $f\colon N\otimes M_1\to M_2$ by $f(n\otimes m) = \varphi(n)(m)$ for $n\in N$ and $m\in M_1$.
Then for $\gamma\in \Gamma$, we have $f(n\gamma\otimes m\gamma) = \varphi(n\gamma)(m\gamma) = \varphi(n)(m)\gamma = f(n\otimes m)\gamma$.
Namely $f$ is $\Gamma$-equivariant.
We define an action of $a\in A$ on $N\otimes M_1$ by $(n\otimes m)a = n\otimes ma$.
Then it coincides with the action of $\Gamma$ on $C[\Gamma']$ and it gives an action of $B$.
This correspondence gives an isomorphism
\[
\Hom_{\Gamma/\Gamma'}(N,\Hom_A(M_1,M_2))\simeq \Hom_B(N\otimes M_1,M_2).
\]
In particular, if $M_2$ is an injective $B$-module, then $\Hom_A(M_1,M_2)$ is an injective $\Gamma/\Gamma'$-module.
Therefore, from $\Hom_B(M_1,M_2) = \Hom_A(M_1,M_2)^{\Gamma/\Gamma'}$, we get a spectral sequence
\[
E_2^{ij} = H^i(\Gamma/\Gamma',\Ext^j_A(M_1,M_2))\Rightarrow \Ext^{i + j}_B(M_1,M_2).
\]
In particular, we have an exact sequence
\begin{equation}\label{eq:ex seq from spectral seq}
0\to H^1(\Gamma/\Gamma',\Hom_A(M_1,M_2))\to \Ext^1_B(M_1,M_2)\to \Ext^1_A(M_1,M_2)^{\Gamma/\Gamma'}
\end{equation}

Moreover, we assume the following situation.
Let $\Gamma_1$ be a finite index subgroup of $\Gamma$ which contains $\Gamma'$ and put $B_1 = A\otimes_{C[\Gamma']}C[\Gamma_1]$.
Then this is a subalgebra of $B$ and $B$ is a free left $B_1$-module with a basis given by a complete representative of $\Gamma_1\backslash\Gamma$.
Assume that $M_1$ has a form $L_1\otimes_{B_1}B$ for some $B_1$-module $L_1$.
We have $M_1 = \bigoplus_{\gamma\in \Gamma_1\backslash\Gamma}L_1\otimes \gamma$.
Since $B$ is flat over $B_1$, we have
\[
\Ext_B^1(M_1,M_2)\simeq \Ext_{B_1}^1(L_1,M_2).
\]

We have a $B_1$-module embedding $L_1\hookrightarrow M_1$.
This is in particular an $A$-homomorphism and we get
\[
\Ext^i_A(M_1,M_2)\to \Ext^1_A(L_1,M_2)
\]
Since $L_1\hookrightarrow M_1$ is a $B_1$-homomorphism, this is a $\Gamma_1$-homomorphism.
Hence this induces
\[
\Ext^i_A(M_1,M_2)\to \Ind_{\Gamma_1}^{\Gamma}(\Ext^i_A(L_1,M_2)).
\]
The decomposition $M_1 = \bigoplus_{\gamma\in \Gamma_1\backslash\Gamma}L_1\otimes \gamma$ respects the $A$-action.
Hence 
\[
\Ext^i_A(M_1,M_2) = \bigoplus_{\gamma\in \Gamma_1\backslash\Gamma} \Ext^i_A(L_1\otimes \gamma,M_2) = \bigoplus_{\gamma\in \Gamma_1\backslash\Gamma} \Ext^i_A(L_1,M_2)\gamma.
\]
Therefore the above homomorphism is an isomorphism
\[
\Ext^i_A(M_1,M_2)\simeq \Ind_{\Gamma_1}^{\Gamma}(\Ext^i_A(L_1,M_2)).
\]
This implies
\begin{gather*}
H^1(\Gamma/\Gamma',\Hom_A(M_1,M_2)) \simeq H^1(\Gamma_1/\Gamma',\Hom_A(L_1,M_2)),\\
\Ext^1_A(M_1,M_2)^{\Gamma/\Gamma'} \simeq \Ext^1_A(L_1,M_2)^{\Gamma_1/\Gamma'}
\end{gather*}
and a commutative diagram
\[
\begin{tikzcd}[column sep=.5cm]
0\arrow[r] & H^1(\Gamma/\Gamma',\Hom_A(M_1,M_2)) \arrow[r]\arrow[d,dash,"\wr"] & \Ext^1_B(M_1,M_2) \arrow[r]\arrow[d,dash,"\wr"] & \Ext^1_A(M_1,M_2)^{\Gamma/\Gamma'}\arrow[d,dash,"\wr"]\\
0\arrow[r] & H^1(\Gamma_1/\Gamma',\Hom_A(L_1,M_2)) \arrow[r] & \Ext^1_{B_1}(L_1,M_2) \arrow[r] & \Ext^1_A(L_1,M_2)^{\Gamma_1/\Gamma'}.
\end{tikzcd}
\]

We also assume that there exists a finite index subgroup $\Gamma_2$ of $\Gamma$ which contains $\Gamma'$ and $M_2 = L_2\otimes_{B_2}B$ where $B_2 = A\otimes_{C[\Gamma']}C[\Gamma_2]$.
Let $\{\gamma_1,\dots,\gamma_r\}$ be a set of complete representatives of $\Gamma_2\backslash\Gamma/\Gamma_1$.
Then the decomposition $M_2 = \bigoplus_{\gamma\in \Gamma_2\backslash \Gamma}L_2\otimes \gamma = \bigoplus_i\bigoplus_{\gamma\in (\Gamma_1\cap \gamma_i^{-1}\Gamma_2\gamma_i)\backslash\Gamma_1}L_2\otimes \gamma_i\gamma$ gives
\[
M_2|_{B_1} = \bigoplus_i L_2\gamma_i\otimes_{B_1\cap \gamma_i^{-1}B_2\gamma_i}B_1,
\]
where $L_2\gamma_i$ is a $\gamma_i^{-1}B_2\gamma_i$-module defined by: $L_2\gamma_i = L_2$ as a vector space and the action is given by $l (\gamma_i^{-1}b\gamma_i) = l\cdot b$ for $l\in L_2$ and $b\in B_2$, here $\cdot$ is the original action of $b\in B_2$ on $L_2$.
From this isomorphism, we get
\[
H^i(\Gamma_1/\Gamma',\Ext^j_A(L_1,M_2)) \simeq \bigoplus_i H^i((\Gamma_1\cap \gamma_i^{-1}\Gamma_2\gamma_i)/\Gamma',\Ext^j_A(L_1,L_2\gamma_i))
\]
and
\[
\Ext^i_{B_1}(L_1,M_2) = \bigoplus_i\Ext^i_{B_1\cap \gamma_i^{-1}B_2\gamma_i}(L_1,L_2\gamma_i)
\]
which is compatible with the exact sequence in \eqref{eq:ex seq from spectral seq}.

Set $A = \mathcal{H}^\aff$, $\Gamma = \Omega(1)$, $\Gamma' = Z_\kappa$, $\Gamma_1 = \Omega(1)_\Xi$ and $\Gamma_2 = \Omega(1)_{\Xi'}$.
Then we get the following lemma.
Recall that $\Omega$ is assumed to be commutative.
\begin{lem}\label{lem:exact seq from spectral seq}
Let $\chi,\chi'$ be characters of $Z_\kappa$ and $J\subset S_{\aff,\chi},J'\subset S_{\aff,\chi'}$.
Put $\Xi = \Xi_{\chi,J}$, $\Xi' = \Xi_{\chi',J'}$ and let $V,V'$ be irreducible $C[\Omega(1)_\Xi]$, $C[\Omega(1)_{\Xi'}]$-modules, respectively.
Let $\{\omega_1,\dots,\omega_r\}$ be a set of complete representatives of $\Omega_\Xi\backslash\Omega/\Omega_{\Xi'} = \Omega/\Omega_{\Xi}\Omega_{\Xi'}$ and define $\Xi'_i$ by $\Xi'_i(X) = \Xi'(\omega_i X\omega_i^{-1})$.
Consider the representation of $C[\Omega(1)_{\Xi'_i}] = \omega_i^{-1}C[\Omega(1)_{\Xi'}]\omega_i$ twisting $V'$ by $\omega_i$ and we denote it by $V'_i$.
Put $\Omega_{\Xi,\Xi'} = \Omega_{\Xi}\cap \Omega_{\Xi'}$ and $\mathcal{H}_{\Xi,\Xi'} = \mathcal{H}_{\Xi}\cap \mathcal{H}_{\Xi'}$.
Then we have a commutative diagram with exact rows:
\[
\begin{tikzcd}
0\arrow[d] & 0\arrow[d]\\
H^1(\Omega,\Hom_{\mathcal{H}^\aff}(\pi_{J,\chi,V},\pi_{J',\chi',V'}))\arrow[d]\arrow[r,dash,"\sim"]  & \displaystyle\bigoplus_i H^1(\Omega_{\Xi,\Xi'_i},\Hom_{\mathcal{H}^\aff}(\Xi\otimes V,\Xi'_i\otimes V'_i))\arrow[d] \\
\Ext^1_{\mathcal{H}}(\pi_{\chi,J,V},\pi_{\chi',J',V'})\arrow[d]\arrow[r,dash,"\sim"] & \displaystyle\bigoplus_i\Ext^1_{\mathcal{H}_{\Xi,\Xi'_i}}(\Xi\otimes V,\Xi'_i\otimes V'_i)\arrow[d] \\
\Ext^1_{\mathcal{H}^\aff}(\pi_{\chi,J,V},\pi_{\chi',J',V'})^{\Omega} \arrow[r,dash,"\sim"]& \displaystyle\bigoplus_i\Ext^1_{\mathcal{H}^\aff}(\Xi\otimes V,\Xi'_i\otimes V'_i)^{\Omega_{\Xi,\Xi'_i}}
\end{tikzcd}\]
\end{lem}

The following theorem will be proved in subsection \ref{subsec:proof of surjectivity, from spectral sequence}.
\begin{thm}\label{thm:surjectivity of exact sequence for supersingulars}
The map $\Ext^1_{\mathcal{H}_{\Xi,\Xi'}}(\Xi\otimes V,\Xi'\otimes V')\to \Ext^1_{\mathcal{H}^\aff}(\Xi\otimes V,\Xi'\otimes V')^{\Omega_{\Xi,\Xi'}}$ is surjective.
\end{thm}

\subsection{$\Ext^1_{\mathcal{H}^\aff}(\Xi\otimes V,\Xi'\otimes V')^{\Omega_{\Xi,\Xi'}}$}\label{subsec:Ext^1,Xi,V,H^aff}
To prove Theorem~\ref{thm:surjectivity of exact sequence for supersingulars}, we analyze $\Ext^1_{\mathcal{H}^\aff}(\Xi\otimes V,\Xi'\otimes V')^{\Omega_{\Xi,\Xi'}}$.
First we have
\[
\Ext^1_{\mathcal{H}^\aff}(\Xi\otimes V,\Xi'\otimes V')\simeq
\Ext^1_{\mathcal{H}^\aff}(\Xi,\Xi')\otimes\Hom_C(V,V')
\]
and the surjective homomorphism $E(\Xi,\Xi')\to \Ext^1_{\mathcal{H}^\aff}(\Xi,\Xi')$.
We have the decomposition $E(\Xi,\Xi') = E_1(\Xi,\Xi')\oplus E_2(\Xi,\Xi')$.
Let $E'_1(\Xi,\Xi')$ (resp.\ $E'_2(\Xi,\Xi')$) be the image of $E_1(\Xi,\Xi')$ (reps.\ $E_2(\Xi,\Xi')$).
By the description of the kernel \eqref{eq:kernel of E->Ext1 (2)}, we have:
\begin{itemize}
\item $\Ext^1_{\mathcal{H}^\aff}(\Xi,\Xi') = E'_1(\Xi,\Xi')\oplus E'_2(\Xi,\Xi')$.
\item $E_1(\Xi,\Xi')\xrightarrow{\sim} E'_1(\Xi,\Xi')$.
\item the dimension of the kernel of $E_2(\Xi,\Xi')\to E'_2(\Xi,\Xi')$ is at most one.
\end{itemize}
Define $E_i\ (i = 2,3)$ by $E_i = E_2(\Xi,\Xi')\cap \bigoplus_{s\in A_i(\Xi,\Xi')}C_s$.
Then $\dim E_2,\dim E_3\le 1$ and $E_i\ne 0$ if and only if for any $s\in A_i(\Xi,\Xi')$ we have $s\chi = \chi'$.

Assume that $s\chi = \chi'$ for any $s\in A_2(\Xi,\Xi')$.
Fix $s_0\in A_2(\Xi,\Xi')$.
Then $a = (a_s)\mapsto a_{s_0}\chi(c_{s_0})^{-1}$ gives an isomorphism $E_2\simeq C$.
Let $\omega\in \Omega_{\Xi,\Xi'}(1)$.
Then 
\[
(a\omega)_{s_0}\chi(c_{s_0})^{-1}
=
a_{\omega s\omega^{-1}}(\omega\widetilde{s_0}\omega^{-1})\chi(c_{\widetilde{s_0}})^{-1}.
\]
Since $\omega$ stabilizes $\chi$, we have $\chi(c_{\widetilde{s_0}}) = (\omega^{-1}\chi)(c_{\widetilde{s_0}}) = \chi(\omega\cdot c_{\widetilde{s_0}}) = \chi(c_{\omega\widetilde{s_0}\omega^{-1}})$.
Therefore we have
\begin{align*}
(a\omega)_{s_0}\chi(c_{s_0})^{-1}
& =
a_{\omega s\omega^{-1}}(\omega \widetilde{s_0}\omega^{-1})\chi(c_{\omega\widetilde{s_0}\omega^{-1}})^{-1}\\
& =
a_{\omega s\omega^{-1}}\chi(c_{\omega s\omega^{-1}})^{-1}\\
& =
a_{s}\chi(c_{s})^{-1}.
\end{align*}
Here the last equality follows from the definition of $E'_2(\Xi,\Xi')$.
Namely $\Omega_{\Xi,\Xi'}$ acts trivially on $E_2$.
By the same argument $\Omega_{\Xi,\Xi'}$ also acts trivially on $E_3$.
Therefore it also acts trivially on $E_2(\Xi,\Xi')$, hence on $E'_2(\Xi,\Xi')$.
Hence
\[
(E'_2(\Xi,\Xi')\otimes\Hom_C(V,V'))^{\Omega_{\Xi,\Xi'}} = E'_2(\Xi,\Xi')\otimes\Hom_{\Omega_{\Xi,\Xi'}}(V,V').
\]


\subsection{Proof of Theorem~\ref{thm:surjectivity of exact sequence for supersingulars}}\label{subsec:proof of surjectivity, from spectral sequence}
Now we prove Theorem~\ref{thm:surjectivity of exact sequence for supersingulars}.
Take $e\in \Ext^1_{\mathcal{H}^\aff}(\Xi\otimes V,\Xi'\otimes V')^{\Omega_{\Xi,\Xi'}}$ and first assume that $e\in E'_1(\Xi,\Xi')$.
Therefore $e$ gives $f_s\in C_s\otimes \Hom_C(V,V')$.
The space $C_s\otimes \Hom_C(V,V')$ is the space of functions $f_s$ on $\{\widetilde{s}\mid \text{$\widetilde{s}$ is a lift of $s$}\}$ with values in $\Hom_C(V,V')$ such that $f(t_1\widetilde{s}t_2) = \chi'(t_1)f(\widetilde{s})\chi(t_2)$ for $t_1,t_2\in Z_\kappa$.
Using this $f_s$, we define an $\mathcal{H}$-module structure on $V\oplus V'$ by
\[
T_{\widetilde{s}} \mapsto \begin{pmatrix} \Xi(T_{\widetilde{s}}) & 0\\ f_s(\widetilde{s}) & \Xi'(T_{\widetilde{s}})\end{pmatrix},\quad
T_\omega \mapsto \begin{pmatrix}V(\omega) & 0\\ 0 & V'(\omega)\end{pmatrix}
\]
where $\widetilde{s}\in S_\aff$, $s$ its image in $S_\aff$ and $f_s = 0$ if $s\notin S_1(\Xi,\Xi')$.
Since $e$ is $\Omega_{\Xi,\Xi'}$-invariant and $E_1(\Xi,\Xi')\to E'_1(\Xi,\Xi')$ is injective, we have $V'(\omega)f_{s}(\widetilde{s})V(\omega^{-1}) = f_{\omega s\omega^{-1}}(\omega\widetilde{s}\omega^{-1})$.
Hence
\begin{align*}
& \begin{pmatrix}V(\omega) & 0\\ 0 & V'(\omega)\end{pmatrix}
\begin{pmatrix} \Xi(T_{\widetilde{s}}) & 0\\ f_s(\widetilde{s}) & \Xi'(T_{\widetilde{s}})\end{pmatrix}
\begin{pmatrix}V(\omega) & 0\\ 0 & V'(\omega)\end{pmatrix}^{-1}\\
& =
\begin{pmatrix} \Xi(T_{\omega\widetilde{s}^{-1}}) & 0\\ f_s(\omega\widetilde{s}\omega^{-1}) & \Xi'(T_{\omega\widetilde{s}\omega^{-1}})\end{pmatrix}.
\end{align*}
Namely the above action gives an action of $\mathcal{H}_{\aff}C[\Omega_{\Xi,\Xi'}]$.
Hence this gives an extension class in $\Ext^1_{\mathcal{H}_{\Xi,\Xi'}}(\Xi\otimes V,\Xi'\otimes V')$ and its image in $\Ext^1_{\mathcal{H}^\aff}(\Xi\otimes V,\Xi'\otimes V')$ corresponds to $e$.

Next we assume that $e$ comes from $E_2$-part.
Then we may assume that there exist $\varphi\in\Hom_{\Omega(1)_{\Xi,\Xi'}}(V,V')$ and $e_0\in E'_2(\Xi,\Xi')$ such that $e$ is given by $e_0\otimes \varphi$.
Take a lift $(a_s)$ of $e_0$ in $E_2(\Xi,\Xi')$ and consider the action of $\mathcal{H}^\aff C[\Omega(1)_{\Xi,\Xi'}]$ on $V\oplus V'$ defined by
\[
T_{\widetilde{s}} \mapsto \begin{pmatrix} \Xi(T_{\widetilde{s}}) & 0\\ a_s(\widetilde{s})\varphi & \Xi'(T_{\widetilde{s}})\end{pmatrix},\quad
T_\omega \mapsto \begin{pmatrix}V(\omega) & 0\\ 0 & V'(\omega)\end{pmatrix}
\]
where $\widetilde{s}\in S_\aff(1)$, $s$ its image in $S_\aff$ and $a_s = 0$ if $s\notin S_2(\Xi,\Xi')$.
Recall that $\Omega(1)_{\Xi,\Xi'}$ acts trivially on $(a_s)$.
Since $\varphi$ is $\Omega(1)_{\Xi,\Xi'}$-equivariant, the calculation as above shows that this gives an action of $\mathcal{H}_{\Xi,\Xi'}$.

\subsection{Calculation of the extensions}
We have
\begin{align*}
\dim\Ext^1_\mathcal{H}(\pi_{\chi,J,V},\pi_{\chi',J',V'}) & = \sum_i \dim\Ext^1_{\mathcal{H}_{\Xi,\Xi'_i}}(\Xi\otimes V,\Xi'_i\otimes V'_i).
\end{align*}
Hence it is sufficient to calculate $\Ext^1_{\mathcal{H}_{\Xi,\Xi'_i}}(\Xi\otimes V,\Xi'_i\otimes V'_i)$.
Now replacing $(\Xi'_i,V_i)$ with $(\Xi',V')$, we explain how to calculate $\Ext^1_{\mathcal{H}_{\Xi,\Xi'}}(\Xi\otimes V,\Xi'\otimes V')$.

Theorem~\ref{thm:surjectivity of exact sequence for supersingulars} implies
\begin{align*}
& \dim\Ext^1_{\mathcal{H}_{\Xi,\Xi'}}(\Xi\otimes V,\Xi'\otimes V')\\
& = \dim H^1(\Omega_{\Xi,\Xi'},\Hom_{\mathcal{H}^\aff}(\Xi\otimes V,\Xi'\otimes V'))
+ \dim \Ext^1_{\mathcal{H}^\aff}(\Xi\otimes V,\Xi'\otimes V')^{\Omega_{\Xi,\Xi'}}.
\end{align*}
Since $\mathcal{H}^\aff$ acts trivially on $V$ and $V$'s, we have
\[
\Hom_{\mathcal{H}^\aff}(\Xi\otimes V,\Xi'\otimes V') = \Hom_{\mathcal{H}^\aff}(\Xi,\Xi')\otimes \Hom_C(V,V')
\]
and it is zero if $\Xi\ne \Xi'$.
If $\Xi = \Xi'$, then 
\[
\Hom_{\mathcal{H}^\aff}(\Xi,\Xi')\otimes \Hom_C(V,V')
=
\Hom_C(V,V')
\]
and hence
\[
H^1(\Omega_{\Xi,\Xi'},\Hom_{\mathcal{H}^\aff}(\Xi\otimes V,\Xi'\otimes V'_i))
\simeq
H^1(\Omega_{\Xi,\Xi'},\Hom_C(V,V')).
\]
This is a group cohomology of an abelian group.

We also have
\[
\Ext^1_{\mathcal{H}^\aff}(\Xi\otimes V,\Xi'\otimes V')
=
\Ext^1_{\mathcal{H}^\aff}(\Xi,\Xi')\otimes\Hom_C(V,V').
\]
and
\[
\Ext^1_{\mathcal{H}^\aff}(\Xi,\Xi')
=E'_1(\Xi,\Xi')\oplus E'_2(\Xi,\Xi').
\]
As we saw in subsection~\ref{subsec:Ext^1,Xi,V,H^aff}, $\Omega_{\Xi,\Xi'}$ acts trivially on $E'_2(\Xi,\Xi')$.
Hence
\[
(E'_2(\Xi,\Xi')\otimes \Hom_C(V,V'))^{\Omega_{\Xi,\Xi'}}
=
E'_2(\Xi,\Xi')\otimes \Hom_{\Omega_{\Xi,\Xi'}(1)}(V,V')
\]
and it is not difficult to calculate this.

Finally we consider $(E'_1(\Xi,\Xi')\otimes \Hom_C(V,V'))^{\Omega_{\Xi,\Xi'}}$.
By Proposition~\ref{prop:extension between supersingulars, for H_aff}, We have $E'_1(\Xi,\Xi')\simeq E_1(\Xi,\Xi') = \bigoplus_{s\in S_1(\Xi,\Xi')}C_s$.
Fix $s_0\in S_1(\Xi,\Xi')$ and let $\Omega(1)_{\Xi,\Xi',s_0}$ be the stabilizer of $s_0$ in $\Omega(1)_{\Xi,\Xi'}$.
Then $C_{s_0}$ is an $\Omega(1)_{\Xi,\Xi',s_0}$-representation.
Consider an $\Omega(1)_{\Xi,\Xi'}$-orbit $\mathcal{S}\subset S_1(\Xi,\Xi')$.
The subspace $\bigoplus_{s\in \mathcal{S}}C_s$ is $\Omega(1)_{\Xi,\Xi'}$-stable and we have an isomorphism
\[
\bigoplus_{s\in \mathcal{S}}C_s \simeq \Ind_{\Omega(1)_{\Xi,\Xi',s_0}}^{\Omega(1)}C_{s_0}
\]
defined by
\[
(a_s)\mapsto (\omega\mapsto (\widetilde{s}_0\mapsto a_{\omega^{-1} s \omega})(\omega^{-1}\widetilde{s}_0\omega)).
\]
Let $\{s_1,\dots,s_r\}$ be a complete representative of the $\Omega(1)_{\Xi,\Xi'}$-orbits in $S_1(\Xi,\Xi')$.
Then we have
\[
E'_1(\Xi,\Xi')\simeq \bigoplus_i \Ind_{\Omega(1)_{\Xi,\Xi',s_i}}^{\Omega(1)_{\Xi,\Xi'}}C_{s_i}.
\]
Hence
\[
(E_1(\Xi,\Xi')\otimes\Hom_C(V,V'))^{\Omega_{\Xi,\Xi'}}
=
\bigoplus_i (C_{s_i}\otimes \Hom_C(V,V'))^{\Omega_{\Xi,\Xi',s_i}}.
\]

\subsection{Example: $G = \mathrm{GL}_n$}
Assume that the data comes from $\mathrm{GL}_n$.
Then the data is as follows, see \cite{MR2122539}.

We have $W_0 = S_n$, $W = S_n\ltimes (F^\times/\mathcal{O}^\times)\simeq S_n\ltimes \Z^n$, $W(1) = S_n\ltimes (F^\times/(1 + (\varpi)))^n = S_n\ltimes (\Z\times \kappa^\times)^n$ and $W^\aff = S_n\ltimes \{(x_i)\in \Z^n\mid \sum x_i = 0\}$.
Set
\[
\omega = \begin{pmatrix}1 & 2 & \cdots & n - 1 & n\\ 2 & 3 & \cdots & n & 1 \end{pmatrix}(0,\dots,0,1)\in S_n\ltimes \Z^n\subset W(1)
\]
and denote its image in $W$ by the same letter $\omega$.
Then $\Omega$ is generated by $\omega$ and $\Omega(1) = \langle \omega\rangle (\kappa^\times)^n$.
We have $\omega^n = (1,\dots,1)$ and it belongs to the center of $W(1)$.
The element $c_{s_i}\in C[Z_\kappa]$ is given by $c_{s_i} = \sum_{t\in \kappa^{\times}}T_{\nu_i(t)\nu_{i + 1}(t)^{-1}}$ where $\nu_i\colon \kappa^\times \to (\kappa^\times)^n$ is an embedding to $i$-th entry and $\nu_{n + 1} = \nu_1$.

Let $\pi_{\chi,J,V}$ and $\pi_{\chi',J',V'}$ be simple supersingular modules and we assume that the dimension of the modules are both $n$.
\begin{rem}
An importance of $n$-dimensional simple supersingular modules is revealed by a work of Grosse-Kl\"onne~\cite{MR3504178}.
He constructed a correspondence between supersingular $n$-dimensional modules of $\mathcal{H}$ and irreducible modulo $p$ $n$-dimensional representations of $\mathrm{Gal}(\overline{F}/F)$.
\end{rem}
We have $\dim \pi_{\chi,J,V} = (\dim V)[\Omega:\Omega_\Xi]$.
Since $\Omega(1)$ is (hence $\Omega(1)_\Xi$ is)  commutative, we have $\dim V = 1$.
Therefore our assumption implies $[\Omega:\Omega_\Xi] = n$.
Since $\omega^n$ is in the center, $\langle \omega^n\rangle\subset \Omega_\Xi$.
Hence $\Omega_\Xi = \langle\omega^n\rangle$ and $\Omega_\Xi = \langle\omega^n\rangle(\kappa^\times)^n$.
Set $\lambda = V(\omega^n)$.
Since $V|_{(\kappa^\times)^n} = \chi$, $V$ is determined by $\chi$ and $\lambda$.
We also put $\lambda' = V'(\omega^n)$.

We define $\chi_j\colon \kappa^\times\to C^\times$ by $\chi(t_1,\dots,t_n) = \chi_1(t_1)\dots \chi_n(t_n)$ and we extend it for any $j\in\Z$ by $\chi_{j\pm n} = \chi_j$.
Then 
\[
(s_i\chi)_j = 
\begin{cases}
\chi_j & (j\ne i,i+1),\\
\chi_{i + 1} & (j = i),\\
\chi_i & (j = i + 1).
\end{cases}
\]
The description of $c_{s_i}$ shows $\chi(c_{s_i}) = 0$ if and only if $\chi_i = \chi_{i + 1}$ if and only if $s_i\chi = \chi$.
Therefore $S_{\aff,\chi} = \{s_i\in S_\aff\mid \chi_i = \chi_{i + 1}\}$.

We consider $\Ext^1_{\mathcal{H}_{\Xi,\Xi'}}(\Xi\otimes V,\Xi'\otimes V')$.
By Theorem~\ref{thm:surjectivity of exact sequence for supersingulars}, we have the exact sequence
\begin{multline*}
0\to H^1(\Omega_{\Xi,\Xi'},\Hom_{\mathcal{H}^\aff}(\Xi,\Xi')\otimes\Hom_C(V,V'))
\to \Ext^1_{\mathcal{H}_{\Xi,\Xi'}}(\Xi\otimes V,\Xi'\otimes V')
\\\to \Ext^1_{\mathcal{H}^\aff}(\Xi\otimes V,\Xi'\otimes V')^{\Omega_{\Xi,\Xi'}}\to 0.
\end{multline*}
The space $\Hom_{\mathcal{H}^\aff}(\Xi,\Xi')$ is $C$ if $\Xi = \Xi'$ and $0$ otherwise.
We have $\Omega_{\Xi,\Xi'} = \langle\omega^n\rangle\simeq \Z$ and $\omega^n$ acts on $\Hom_C(V,V')$ by $\lambda^{-1}\lambda'$.
Therefore $H^1(\Omega_{\Xi,\Xi'},\Hom_C(V,V')) = C$ if $\lambda = \lambda'$ and $0$ otherwise.
Namely we get
\begin{equation}\label{eq:GL_n,first part of Ext}
\dim H^1(\Omega_{\Xi,\Xi'},\Hom_{\mathcal{H}^\aff}(\Xi,\Xi')\otimes\Hom_C(V,V'))
=
\begin{cases}
1 & \Xi = \Xi',V = V',\\
0 & \text{otherwise}.
\end{cases}
\end{equation}

Note that $\Omega_{\Xi,\Xi'}$ acts on $S_\aff$ trivially since $\Omega_{\Xi,\Xi'}$ is in the center of $W$.
Hence the stabilizer of each $s\in S_\aff$ in $\Omega_{\Xi,\Xi'}$ is $\Omega_{\Xi,\Xi'}$ itself and each orbit is a singleton.
Therefore by the previous subsection, we have
\begin{align*}
&\Ext^1_{\mathcal{H}^\aff}(\Xi\otimes V,\Xi'\otimes V')^{\Omega_{\Xi,\Xi'}}\\
&=
\bigoplus_{s\in S_1(\Xi,\Xi')}(C_s\otimes \Hom_C(V,V'))^{\Omega_{\Xi,\Xi'}}
\oplus
E'_2(\Xi,\Xi')\otimes\Hom_{\Omega_{\Xi,\Xi'}}(V,V').
\end{align*}
Since $\omega^n\in \Omega_{\Xi,\Xi',s} = \Omega_{\Xi,\Xi'}$ is in the center of $W(1)$, it acts trivially on $C_s$.
Hence $(C_s\otimes \Hom_C(V,V'))^{\Omega_{\Xi,\Xi',s}} = \Hom_{\Omega_{\Xi,\Xi'}}(V,V')$ and it is not zero if and only if $\lambda = \lambda'$.
Hence
\[
\Ext^1_{\mathcal{H}^\aff}(\Xi\otimes V,\Xi'\otimes V')^{\Omega_{\Xi,\Xi'}}
\simeq
\begin{cases}
C^{S_1(\Xi,\Xi')}\oplus E_2'(\Xi,\Xi') & \lambda = \lambda',\\
0 & \text{otherwise.}
\end{cases}
\]

A complete representative of $\Omega/\Omega_{\Xi}\Omega_{\Xi'}$ is given by $\{1,\omega,\dots,\omega^{n - 1}\}$.
Put $\Xi'_i = \Xi'\omega^i$.
This is parametrized by $(\chi\omega^i,J_i = \omega^iJ\omega^{-i})$.
We have $(\chi\omega^i)_j = \chi_{j + i}$ and $\omega^iJ\omega^{-i} = \{s_{j + i}\mid s_j\in J\}$.
We put $V'_i = V'\omega^i$.
Then $V'_i(\omega) = V'(\omega)$ and $V'_i|_{Z_\kappa} = \chi\omega^i$.

The cohomology group $H^1(\Omega_{\Xi,\Xi'_i},\Hom_{\mathcal{H}^\aff}(\Xi,\Xi'_i)\otimes\Hom_C(V,V'_i))$ is zero if and only if $(\Xi,V)\ne (\Xi'_i,V'_i)$ by \eqref{eq:GL_n,first part of Ext}.
There exists at most one $i$ such that $(\Xi,\Xi'_i)\ne (V,V'_i)$ and such $i$ exists if and only if $(\Xi,V)$ is $\Omega$-conjugate to $(\Xi',V')$.
Hence
\begin{align*}
&\dim \bigoplus_{i = 0}^{n - 1}H^1(\Omega_{\Xi,\Xi'_i},\Hom_{\mathcal{H}^\aff}(\Xi,\Xi'_i)\otimes\Hom_C(V,V'_i))\\
&=
\begin{cases}
1 & \text{$(\Xi,V)$ is $\Omega$-conjugate to $(\Xi',V')$},\\
0 & \text{otherwise}.
\end{cases}
\end{align*}
We also have
\begin{align*}
&\dim\bigoplus_{i = 0}^{n - 1}\Ext^1_{\mathcal{H}^\aff}(\Xi\otimes V,\Xi'\otimes V')^{\Omega_{\Xi,\Xi'}}\\
&\simeq
\begin{cases}
\sum_{i = 0}^{n - 1}(\#S_1(\Xi,\Xi'_i) + \dim E_2'(\Xi,\Xi'_i)) & \lambda = \lambda',\\
0 & \text{otherwise.}
\end{cases}
\end{align*}
and each term can be calculated by the description in Proposition~\ref{prop:extension between supersingulars, for H_aff}.

\subsection{$\mathrm{GL}_2$}
Now we assume $n = 2$ and we compute $\Ext^1_{\mathcal{H}}(\pi_{\chi,J,V},\pi_{\chi',J',V'})$.
We continue to use the notation in the previous subsection.
Then $\omega$ switches $s_0$ and $s_1$.
\begin{lem}\label{lem:only self-extension when GL_2}
The non-vanishing of $\Ext^1_{\mathcal{H}}(\pi_{\chi,J,V},\pi_{\chi',J',V'})$ implies that $(\chi,J,V)$ is conjugate to $(\chi',J',V')$ by $\Omega$.
\end{lem}
\begin{proof}
As we have seen in the above, non-vanishing of $\Ext^1$ implies $V(\omega) = V'(\omega) = (V'\omega)(\omega)$.
Hence it is sufficient to prove that $(\chi,J')$ is conjugate to $(\chi',J')$.

If $H^1(\Omega_{\Xi,\Xi'},\Hom_{\mathcal{H}^\aff}(\Xi\otimes \Xi')\otimes\Hom_C(V,V'))\ne 0$, we have $\Xi = \Xi'$ and $V = V'$.
Hence we have the lemma.

If $\Ext^1_{\mathcal{H}^\aff}(\Xi\otimes V,\Xi'\otimes V')^{\Omega_{\Xi,\Xi'}}\ne 0$, then $C_s\ne 0$, hence $\chi' = s\chi$ for some $s\in S_\aff$.
Since we assume $G = \mathrm{GL}_2$, $s_0\chi = s_1\chi = \chi\omega$.
Since $\pi_{\chi,J,V}$ and $\pi_{\chi',J',V'}$ are both supersingular, the possibility of $(J,J')$ is $(\emptyset,\emptyset)$, $(\{s_0\},\{s_1\})$, $(\{s_0\},\{s_1\})$, $(\{s_0\},\{s_0\})$, $(\{s_1\},\{s_1\})$ and except the last two cases, we have $J = \omega J'\omega^{-1}$.
If $J = J' = \{s_0\}$, then $s_0\in S_{\aff,\chi}$, hence $s_0\chi = \chi$.
Since $s_0\chi = s_1\chi$, we have $S_{\aff,\chi} = S_\aff$.
Hence $S_1(\Xi,\Xi') = \emptyset$.
We also have $A_2(\Xi,\Xi') = A_3(\Xi,\Xi') = \emptyset$.
Therefore we get $\Ext^1_{\mathcal{H}^\aff}(\Xi\otimes V,\Xi'\otimes V')^{\Omega_{\Xi,\Xi'}} = 0$.
By the same way, if $J = J' = \{s_1\}$ then $\Ext^1_{\mathcal{H}^\aff}(\Xi\otimes V,\Xi'\otimes V')^{\Omega_{\Xi,\Xi'}} = 0$.
\end{proof}
Since $\pi_{\chi,J,V}$ only depends on the $\Omega$-orbit of $(\chi,J,V)$, we may assume $(\chi,J,V) = (\chi',J',V')$.
In this case, $H^1(\Omega_{\Xi,\Xi'},\Hom_{\mathcal{X}^\aff}(\Xi\otimes V,\Xi_i\otimes V_i))$ is one-dimensional if $i = 0$ and zero if $i = 1.$

\begin{enumerate}
\item The case of $\chi_1 = \chi_2$.
Then we have $S_{\aff,\chi} = S_\aff$.
By the proof of Lemma~\ref{lem:only self-extension when GL_2}, we have $\Ext^1_{\mathcal{H}^\aff}(\Xi\otimes V,\Xi\otimes V) = 0$.
We have $S_1(\Xi,\Xi_1)= \emptyset$, $A_2(\Xi,\Xi_1) = J_1 = \omega J\omega^{-1}$ and $A_3(\Xi,\Xi_1) = J_0 = J$.
Hence the description in Proposition~\ref{prop:extension between supersingulars, for H_aff} shows that $\dim E'_2(\Xi,\Xi') = 1$ and hence $\dim\Ext^1_{\mathcal{H}^\aff}(\Xi\otimes V,\Xi\otimes V) = 1$.

\item The case of $\chi_1 \ne \chi_2$.
Then we have $S_{\aff,\chi} = \emptyset$.
Since $\chi\ne s\chi = s\chi_0$ for $s = s_0,s_1$, $C_s = 0$.
Therefore $\Ext^1_{\mathcal{H}^\aff}(\Xi\otimes V,\Xi_0\otimes V_0) = 0$.
Since $S_{\aff,\chi} = \emptyset$, $\Xi(T_s) = \Xi'(T_s) = 0$ for any $s\in S_\aff$.
We have $A_2(\Xi,\Xi_1) = A_3(\Xi,\Xi_1) = \emptyset$ and $S_1(\Xi,\Xi_1) = S_\aff$.
Therefore $E'_1(\Xi,\Xi_1) = 0$ and $\dim E'_2(\Xi,\Xi_1) = \#S_1(\Xi,\Xi_1) = \#S_\aff = 2$.
\end{enumerate}
Hence we have
\[
\dim\Ext^1_{\mathcal{H}}(\pi_{\chi,J,V},\pi_{\chi',J',V'}) = 
\begin{cases}
0 & (\pi_{\chi,J,V}\not\simeq\pi_{\chi',J',V'}),\\
2 & (\pi_{\chi,J,V}\simeq\pi_{\chi',J',V'},\ \chi_1 = \chi_2),\\
3 & (\pi_{\chi,J,V}\simeq\pi_{\chi',J',V'},\ \chi_1 \ne \chi_2).
\end{cases}
\]
This recovers \cite[Corollary~6.7]{MR2931521}.
(Note that in \cite{MR2931521}, they calculate the extensions with fixed central character. Since we do not fix the central character here, the dimension calculated here is one greater than the dimension they calculated.)


\begin{thebibliography}{AHHV17}
\bibitem[Abe]{arXiv:1406.1003_accepted}
N.~Abe, \emph{{M}odulo {$p$} parabolic induction of pro-{$p$}-{I}wahori {H}ecke algebra}, J. Reine Angew. Math., DOI:10.1515/crelle-2016-0043.

\bibitem[Abe16]{arXiv:1612.01312}
N.~Abe, \emph{{P}arabolic inductions for pro-{$p$}-{I}wahori {H}ecke algebras}, arXiv:1612.01312.

\bibitem[Abe17]{arXiv:1704.00407}
N.~Abe, \emph{{I}nvolutions on pro-{$p$}-{I}wahori {H}ecke algebras}, arXiv:1704.00407.

\bibitem[AHHV17]{MR3600042}
N.~Abe, G.~Henniart, F.~Herzig, and M.-F. Vign\'eras, \emph{A classification of irreducible admissible mod {$p$} representations of {$p$}-adic reductive groups}, J. Amer. Math. Soc. \textbf{30} (2017), no.~2, 495--559.

\bibitem[AHV17]{arXiv:1703.10384}
N.~Abe, G.~Henniart, and M.-F. Vign\'eras, \emph{{O}n pro-{$p$}-{I}wahori invariants of {$R$}-representations of reductive {$p$}-adic groups}, arXiv:1703.10384.

\bibitem[BP12]{MR2931521}
C.~Breuil and V.~Pa{\v{s}}k{\=u}nas, \emph{Towards a modulo {$p$} {L}anglands correspondence for {${\rm GL}_2$}}, Mem. Amer. Math. Soc. \textbf{216} (2012), no.~1016, vi+114.

\bibitem[Fay05]{MR2134290}
M.~Fayers, \emph{0-{H}ecke algebras of finite {C}oxeter groups}, J. Pure Appl. Algebra \textbf{199} (2005), no.~1-3, 27--41.

\bibitem[GK16]{MR3504178}
E.~Grosse-Kl{\"o}nne, \emph{From pro-{$p$} {I}wahori--{H}ecke modules to {$(\varphi,\Gamma)$}-modules, {I}}, Duke Math. J. \textbf{165} (2016), no.~8, 1529--1595.

\bibitem[Koz15]{arXiv:1512.00247}
K.~Koziol, \emph{{H}omological dimension of simple pro-p-{I}wahori--{H}ecke modules}, arXiv:1512.00247.

\bibitem[Nad17]{arXiv:1703.03110}
S.~Nadimpalli, \emph{{O}n extensions of characters of affine pro-{$p$} {I}wahori--{H}ecke algebra}, arXiv:1703.03110.

\bibitem[Oll14]{MR3263136}
R.~Ollivier, \emph{Compatibility between {S}atake and {B}ernstein isomorphisms in characteristic {$p$}}, Algebra Number Theory \textbf{8} (2014), no.~5, 1071--1111.

\bibitem[OS14]{MR3249689}
R.~Ollivier and P.~Schneider, \emph{Pro-{$p$} {I}wahori-{H}ecke algebras are {G}orenstein}, J. Inst. Math. Jussieu \textbf{13} (2014), no.~4, 753--809.

\bibitem[Pa{\v{s}}10]{MR2667891}
V.~Pa{\v{s}}k{\=u}nas, \emph{Extensions for supersingular representations of {${\rm GL}_2(\mathbb{Q}_p)$}}, Ast\'erisque (2010), no.~331, 317--353.

\bibitem[Vig05]{MR2122539}
M.-F. Vign{\'e}ras, \emph{Pro-{$p$}-{I}wahori {H}ecke ring and supersingular {$\overline{\mathbf F}_p$}-representations}, Math. Ann. \textbf{331} (2005), no.~3, 523--556.

\bibitem[Vig15a]{Vigneras-prop-III}
M.-F. Vign{\'e}ras, \emph{The pro-{$p$}-{I}wahori {H}ecke algebra of a {$p$}-adic group {III}}, J. Inst. Math. Jussieu (2015), 1--38.

\bibitem[Vig15b]{MR3437789}
M.-F. Vign{\'e}ras, \emph{The pro-p {I}wahori {H}ecke algebra of a reductive p-adic group, {V} (parabolic induction)}, Pacific J. Math. \textbf{279} (2015), no.~1-2, 499--529.

\bibitem[Vig16]{MR3484112}
M.-F. Vign{\'e}ras, \emph{The pro-{$p$}-{I}wahori {H}ecke algebra of a reductive {$p$}-adic group {I}}, Compos. Math. \textbf{152} (2016), no.~4, 693--753.

\end{thebibliography}
\end{document}